\documentclass[11pt]{amsart}

\usepackage{subfiles}

\usepackage[normalem]{ulem}
\usepackage{upgreek}

\usepackage{tabu}
\usepackage[margin=1in]{geometry}
\usepackage[dvipsnames]{xcolor}   		             		
\usepackage{graphicx}		
\usepackage{colonequals}	
\usepackage{amssymb}
\usepackage{mathrsfs}
\usepackage{amsthm}
\usepackage{amsmath}
\usepackage{stmaryrd}

\usepackage{tikz}
\usepackage{tikz-cd}
\usetikzlibrary{arrows,arrows.meta}
\usetikzlibrary{arrows.meta,        
                decorations.pathmorphing,
                matrix}%
\tikzset{
  commutative diagrams/.cd, 
  arrow style=tikz, 
  diagrams={>=cm to}
}

\tikzset{
  c/.style={every coordinate/.try}
}

\makeatletter
\newcommand{\subalign}[1]{%
  \vcenter{%
    \Let@ \restore@math@cr \default@tag
    \baselineskip\fontdimen10 \scriptfont\tw@
    \advance\baselineskip\fontdimen12 \scriptfont\tw@
    \lineskip\thr@@\fontdimen8 \scriptfont\thr@@
    \lineskiplimit\lineskip
    \ialign{\hfil$\m@th\scriptstyle##$&$\m@th\scriptstyle{}##$\hfil\crcr
      #1\crcr
    }%
  }%
}
\makeatother

\usepackage{accents}
\usepackage{upgreek}
\usepackage{enumerate}
\usepackage{bm}
\usepackage{mathtools}
\usepackage[all]{xy}
\usepackage{caption}

\usepackage{url}
\usepackage{float}
\usepackage{todonotes}
\usepackage{bbm}

\usepackage{longtable}

\usepackage[full]{textcomp}
\usepackage[cal=cm]{mathalfa}
\usepackage{xparse}
\usepackage{comment}
\usepackage{cite}

\theoremstyle{theorem}
\newtheorem{innercustomthm}{Theorem}
\newenvironment{customthm}[1]
  {\renewcommand\theinnercustomthm{#1}\innercustomthm}
  {\endinnercustomthm}
  
\makeatletter
\def\@tocline#1#2#3#4#5#6#7{\relax
  \ifnum #1>\c@tocdepth 
  \else
    \par \addpenalty\@secpenalty\addvspace{#2}%
    \begingroup \hyphenpenalty\@M
    \@ifempty{#4}{%
      \@tempdima\csname r@tocindent\number#1\endcsname\relax
    }{%
      \@tempdima#4\relax
    }%
    \parindent\z@ \leftskip#3\relax \advance\leftskip\@tempdima\relax
    \rightskip\@pnumwidth plus4em \parfillskip-\@pnumwidth
    #5\leavevmode\hskip-\@tempdima
      \ifcase #1
       \or\or \hskip 1em \or \hskip 2em \else \hskip 3em \fi%
      #6\nobreak\relax
    \dotfill\hbox to\@pnumwidth{\@tocpagenum{#7}}\par
    \nobreak
    \endgroup
  \fi}
\makeatother

\usetikzlibrary{calc}
\usetikzlibrary{fadings}
\usetikzlibrary{decorations.pathmorphing}
\usetikzlibrary{decorations.pathreplacing}

\newcommand{\Aut}{\operatorname{Aut}}

\newcounter{marginnote}
\DeclareMathAlphabet{\mathpzc}{OT1}{pzc}{m}{it}

\usepackage[pagebackref]{hyperref}
\hypersetup{
  colorlinks   = true,          
  urlcolor     = blue,          
  linkcolor    = purple,          
  citecolor   = blue,             
}

\theoremstyle{theorem}
\newtheorem{theorem}{Theorem}[section]

\newtheorem{lemma}[theorem]{Lemma}
\newtheorem{proposition}[theorem]{Proposition}

\theoremstyle{definition}
\newtheorem{remark}[theorem]{Remark}

\newtheorem*{runningexample*}{Running example}

\newtheorem*{aside*}{Aside}

\newtheorem{construction}[theorem]{Construction}

\newtheorem{definition}[theorem]{Definition}
\newtheorem{example}[theorem]{Example}

\newtheorem{proposition-definition}[theorem]{Proposition-Definition}

\DeclareMathOperator{\Hom}{Hom}

\DeclareMathOperator{\Pic}{Pic}
\DeclareMathOperator{\Cl}{Cl}

\newcommand{\RR}{\mathbb{R}}

\newcommand{\Star}{\operatorname{St}}

\newcommand{\f}{\mathsf{f}}

\newcommand{\Gm}{\mathbb{G}_{\operatorname{m}}}

\newcommand{\op}[1]{\operatorname{#1}}

\newcommand{\bcd}{\begin{center}\begin{tikzcd}}
\newcommand{\ecd}{\end{tikzcd}\end{center}}

\newcommand{\Aaff}{\mathbb{A}}

\newcommand{\PP}{\mathbb{P}}
\newcommand{\ZZ}{\mathbb{Z}}

\newcommand{\OO}{\mathcal{O}}
\newcommand{\N}{\mathbb{N}}
\newcommand{\Z}{\mathbb{Z}}
\newcommand{\Q}{\mathbb{Q}}

\newcommand{\R}{\mathbb{R}}

\newcommand{\Speck}{\operatorname{Spec}\kfield}
\newcommand{\kfield}{\Bbbk}

\newcommand{\Proj}{\operatorname{Proj}}

\newcommand{\Acal}{\mathcal{A}}
\newcommand{\Dcal}{\mathcal{D}}
\newcommand{\Xcal}{\mathcal{X}}
\newcommand{\Bcal}{\mathcal{B}}

\newcommand{\pt}{\operatorname{pt}}

\newcommand{\rtrop}{\mathsf{r}}
\newcommand{\ptrop}{\mathsf{p}}

\newcommand{\p}{\mathsf{p}}

\newcommand{\Spec}{\operatorname{Spec}}
\newcommand{\acts}{\curvearrowright}

\newcommand{\red}[1]{\color{red} {#1} \color{black}}

\NewDocumentCommand{\compatibilitydatum}{m m m m m m O{} O{} O{}}{
\begin{equation*} \begin{tikzcd}[ampersand replacement=\&]
  \: \arrow{r} \& {#1} \arrow{r} \arrow{d}{#7} \& {#2} \arrow{r} \arrow{d}{#8} \& {#3} \arrow{r}{[1]} \arrow{d}{#9} \& \: \\
  \: \arrow{r} \& {#4} \arrow{r} \& {#5} \arrow{r} \& {#6} \arrow{r} \& \:
\end{tikzcd} \end{equation*}}

\NewDocumentCommand{\commutingsquare}{m m m m o O{} O{} O{} O{}}{
\begin{equation}\begin{tikzcd}[ampersand replacement=\&] \label{#5}
  #1 \arrow{r}{#6} \arrow{d}{#7} \& #2 \arrow{d}{#8} \\
  #3 \arrow{r}{#9} \& #4
\end{tikzcd}\IfValueTF{#5}{\label{#5}}{} \end{equation}}

\NewDocumentCommand{\cartesiansquare}{m m m m O{} O{} O{} O{}}{
\begin{equation*}\begin{tikzcd}[ampersand replacement=\&]
  #1 \arrow{r}{#5} \arrow{d}{#6} \arrow[dr, phantom, "\square"] \& #2 \arrow{d}{#7} \\
  #3 \arrow{r}{#8} \& #4
\end{tikzcd} \end{equation*}}

\NewDocumentCommand{\cartesiansquarelabel}{m m m m m O{} O{} O{} O{}}{
\begin{tikzcd}[ampersand replacement=\&]
  #1 \arrow{r}{#6} \arrow{d}{#7} \arrow[dr, phantom, "\square"] \& #2 \arrow{d}{#8} \\
  #3 \arrow{r}{#9} \& #4
\end{tikzcd}\IfValueTF{#5}{\label{#5}}{}
}

\NewDocumentCommand{\triangleofspaces}{m m m O{} O{} O{}}{
\begin{tikzcd} [ampersand replacement=\&]
#1 \arrow{r}{#4} \arrow[bend right]{rr}{#5} \& #2 \arrow{r}{#6} \& #3
\end{tikzcd}}

\begin{document}
 
\title{Tropical expansions and toric variety bundles}
\author{Francesca Carocci and Navid Nabijou}

\begin{abstract}
A tropical expansion is a degeneration of a toroidal embedding, induced by a polyhedral subdivision of its tropicalisation. Each irreducible component of a tropical expansion admits a collapsing map down to a stratum of the original variety. We study the relative geometry of this map. We give a polyhedral criterion for the map to have the structure of a toric variety bundle, and prove that this structure always exists over the interior of the codomain. We give examples demonstrating that this is the strongest statement one can hope for in general. In addition, we provide a combinatorial recipe for constructing the toric variety bundle as a fibrewise GIT quotient of an explicit split vector bundle. Our proofs make systematic use of Artin fans as a language for globalising local toric models.
\end{abstract}

\maketitle
\tableofcontents

\section*{Introduction}

\noindent Given a toroidal embedding $(X|D)$, a tropical expansion is a multi-parameter degeneration of the pair induced by a polyhedral subdivision of the tropicalisation. In recent years these have come to occupy a central role in logarithmic enumerative geometry and moduli theory \cite{NishinouSiebert,MandelRuddat,DhruvExpansions,MR20}. 

\subsection{Results} In this paper we establish structural results on the geometry of tropical expansions, and develop a combinatorial toolkit for working with them in practice.

The irreducible components of a tropical expansion are indexed by the vertices of the corresponding polyhedral subdivision. Given such a vertex $v$, the component $Y_v$ admits a natural collapsing map down to a stratum $X_v$ of the original pair
\[ \rho_v \colon Y_v \to X_v.\]
Since $(X|D)$ is arbitrary the stratum $X_v$ is equally arbitrary, and nothing particular can be said about it. Instead, our goal is to describe the relative geometry of $\rho_v$.

The initial hope is that $\rho_v$ is a toric variety bundle, i.e. a fibrewise equivariant compactification of a principal torus bundle. Unfortunately, this is not always the case: $\rho_v$ can have reducible fibres, or can even fail to be flat.\footnote{These issues do not arise for expansions of smooth pairs. This accounts for much of the simplicity of relative enumerative geometry when compared with its logarithmic counterpart.} Examples are presented in Section~\ref{sec: examples}.

Our first main result shows that $\rho_v$ is a toric variety bundle over the interior $X_v^\circ$ of the stratum $X_v$. We also provide a combinatorial recipe for constructing this bundle.

\begin{customthm}{A}[Theorem~\ref{thm: Yvbullet description}] \label{thm: Yvbullet introduction} Define $Y_v^\bullet \colonequals \rho_v^{-1}(X_v^\circ)$. Then the restricted collapsing morphism 
\[ \rho_v \colon Y_v^\bullet \to  X_v^\circ \]
is a toric variety bundle. It can be constructed as the fibrewise GIT quotient of a direct sum of explicit line bundles on $X_v^\circ$.
\end{customthm}

This gives a complete description of the open subvariety
\[ Y_v^\bullet \hookrightarrow Y_v.\]
If the stratum $X_v$ is minimal, then $X_v^\circ=X_v$ and so $Y_v^\bullet=Y_v$. This is often the case in applications, e.g. over torus-fixed loci in the moduli space of expanded stable maps \cite{DhruvExpansions}. When $X_v$ is not minimal, Theorem~\ref{thm: Yvbullet introduction} instead gives a cut-and-paste description of $Y_v \to X_v$ as a stratified union of toric variety bundles (Remark~\ref{rmk: cut and paste}).

Our second main result is a criterion determining when the above description extends to the entire component. This generalises a well-known criterion in toric geometry \cite[Theorem~3.3.19]{CLS}.
\begin{customthm}{B}[Theorem~\ref{thm: Yv description}] \label{thm: Yv introduction}
Let $\Upsilon/\omega_v$ and $\Sigma/\sigma_v$ be the tropicalisations of $Y_v$ and $X_v$ respectively, and let $\Phi_v \hookrightarrow \Upsilon/\omega_v$ be the fibre fan constructed in Definition~\ref{def: fibre fan}. The collapsing morphism $\rho_v \colon Y_v \to X_v$ is a toric variety bundle if there is an embedding of cone complexes
\[ \Sigma/\sigma_v \to \Upsilon/\omega_v \]
which is a section of the projection $\Upsilon/\omega_v \to \Sigma/\sigma_v$, and is such that $\Upsilon/\omega_v$ is the internal sum:
\[ \Upsilon/\omega_v = \Sigma/\sigma_v + \Phi_v. \]
In this case, $Y_v$ is also a fibrewise GIT quotient of a direct sum of line bundles on $X_v$. When $X$ is smooth there is a combinatorial algorithm expressing these line bundles in terms of piecewise-linear functions on tropicalisations (Section~\ref{sec: determining the mixing collection for Yv}).
\end{customthm}
These are the strongest statements one can hope for in general. The morphism $Y_v \to X_v$ can be an arbitrary toroidal modification of a toric variety bundle. Given this, Theorem~\ref{thm: Yvbullet introduction} allows us to completely describe its behaviour over the interior, and Theorem~\ref{thm: Yv introduction} provides a criterion for this description to persist over the boundary.

As necessary background, in Section~\ref{sec: toric variety bundles} we study toric variety bundles. We establish several useful facts, including equivalences between three complementary approaches.

\subsection{Context} The study of tropical expansions has a long history. They were introduced, for toric pairs $(X|D)$ and with one-dimensional base cone, in \cite{NishinouSiebert}. In this setting they form an important special case of the Mumford degeneration \cite{MumfordAbelian}.

It has long been known that this construction extends beyond the toric case to the setting of toroidal embeddings. Despite this, there has been to our knowledge no systematic treatment of the geometry of such expansions, and in particular no general recipe for constructing $Y_v$ from $X_v$. Whereas in the toric setting all such components are again toric varieties with explicitly determined fans, this is not true in general, since the geometry of $(X|D)$ is unconstrained.

 One of the fundamental insights of \cite{DhruvExpansions,MR20} is that it is sensible and profitable to consider tropical expansions in a universal fashion, i.e. over arbitrary base cones. This is a crucial step in the construction of logarithmic Gromov--Witten and Donaldson--Thomas invariants via expanded targets. Over arbitrary base cones, flatness of the degenerating family becomes much more delicate \cite{AbramovichKaru,MolchoSS}.
 
 The present paper is a companion to \cite{CN21}, which describes the higher-rank rubber torus acting on the tropical expansions which appear as targets in the moduli space of expanded stable maps. Taken together, these works provide an avenue for probing the recursive structure of these moduli spaces.



\subsection{Acknowledgements} We have benefited from numerous conversations with Luca~Battistella, Sam~Molcho, Leonid~Monin, and Dhruv~Ranganathan. Heartfelt thanks are owed to the anonymous referee, who supplied several key suggestions helping us to remove a superfluous smoothness hypothesis in Sections~\ref{sec: Yvbullet} and \ref{sec: Yv closed stratum}, and identified an important conceptual mistake in an earlier version of the proof of Theorem~\ref{thm: Yv description}.

 F.C. was supported by the EPFL Chair of Arithmetic Geometry (ARG), and N.N. was supported by the Herchel Smith Fund.

\section{Toroidal preliminaries} \label{sec: toroidal preliminaries}

\noindent We review the theory of toroidal embeddings, with a heavy bias towards tropicalisations and Artin fans. For further reading, see \cite{KKMSD} and \cite[Section~1]{AbramovichKaru}. We assume familiarity with the basics of toric varieties \cite{FultonToric,CLS}.

\subsection{Toroidal embeddings}\label{sec:toroidal}
Fix a normal variety $X$ and a reduced divisor $D \subseteq X$.

\begin{definition} The pair $(X|D)$ is a \textbf{toroidal embedding} (without self-intersections) if, Zariski-locally, there exists a strictly convex rational polyhedral cone $\sigma$ and a smooth morphism
\begin{equation} \label{eqn: toric model} p \colon (X|D)|_U \to (U_\sigma |\, \partial U_\sigma) \end{equation}
such that $p^{-1}(\partial U_\sigma)=D|_U$. Here $U_\sigma= \Speck [S_\sigma]$ is the associated affine toric variety and $\partial U_\sigma \subseteq U_\sigma$ is the toric boundary. The local model $U_\sigma$ is determined, up to torus factors, by the open set $U$ (however, there are choices for the morphism $p \colon U \to U_\sigma$).

A toroidal morphism $f \colon (X|D)\to (Y|E)$ is a morphism $f\colon X\to Y$ such that $f^{-1}(E) \subseteq D$ set-theoretically. Locally, a toroidal morphism restricts to a toric morphism on toric models.
\end{definition}

\begin{remark} An important class of toroidal embeddings are the simple normal crossings pairs. These consist of a smooth variety $X$ and a divisor $D = D_1+\ldots+D_k \subseteq X$ with the $D_i$ smooth and intersecting transversely. The local model around the intersection of $l$ divisor components is $\Aaff^l$.
\end{remark}

\begin{remark}
The divisor $D$ defines a stratification of $X$. When $D=D_1+\ldots+D_k$ is simple normal crossings, the locally-closed strata are precisely the connected components of the loci
\[D_I^\circ = \bigcap_{i \in I} D_i \setminus \bigcup_{i\notin I} D_i \]
for arbitrary $I \subseteq \{1,\ldots,k\}$.
\end{remark}

\begin{remark} Toroidal embeddings are usually defined with respect to the \'etale topology, which permits irreducible components of $D$ with self-intersections. We restrict to the Zariski topology to exclude such self-intersections, ensuring that every closed stratum of $(X|D)$ is again a toroidal embedding. This is necessary to describe the relative geometry of tropical expansions. Self-intersections can always be removed via toroidal modifications (see Section~\ref{sec: toroidal modification}).
\end{remark}

\subsection{Tropicalisations} A \textbf{cone} is a pair
\[ \sigma = (\sigma_\R, N_\sigma) \]
where $N_\sigma$ is a lattice and $\sigma_{\R} \subseteq N_\sigma \otimes \RR$ is a strictly convex rational polyhedral cone. A \textbf{map} of cones $\sigma_1 \to \sigma_2$ is a lattice morphism $N_{\sigma_1} \to N_{\sigma_2}$ sending $\sigma_{1\R}$ into $\sigma_{2\R}$. A \textbf{face map} is a map such that $N_{\sigma_1} \to N_{\sigma_2}$ is injective with saturated image, and the image of $\sigma_{1\R}$ is a face of $\sigma_{2\R}$.

An abstract \textbf{cone complex} is a collection of cones glued together along faces \cite[Section~2]{AbramovichCaporasoPayne}. Formally, it is a diagram of cones connected by face maps, such that every face of every cone appears as the image of a face map in the diagram. We prohibit non-identity isomorphisms, and the presence of more than one face map between a given pair of cones.

To every toroidal embedding we associate an abstract cone complex: the tropicalisation. This generalises the fan of a toric variety.

\begin{definition}\label{tropicalisation}
The \textbf{tropicalisation} $\Sigma(X|D)$ is the following cone complex:
\begin{itemize}
\item Cones: Every locally-closed stratum $S \subseteq X$ has an associated cone $\sigma_S$ corresponding to the local toric model of $(X|D)$ in a Zariski neighbourhood of $S$.\footnote{If the local toric model has torus factors, these are removed by quotienting by the subtorus of the big torus which acts freely. This is equivalent to restricting to the lattice spanned by the cone $\sigma_S$, and ensures uniqueness of the local model. See also Remark~\ref{rmk: defn of sigmaS} for an alternative description.} We take the collection of all such cones $\sigma_S$.\medskip
\item Face maps: If $S_2 \subseteq \overline{S_1}$ then there is an associated face map $\sigma_{S_1} \hookrightarrow \sigma_{S_2}$. We take the collection of all such face maps.
\end{itemize}
\end{definition}

Cones $\sigma \in \Sigma(X|D)$ correspond bijectively to strata of $(X|D)$. This correspondence is inclusion-reversing, and the dimension of a cone is the codimension of the stratum. 

If $(X|D)$ is a simple normal crossings pair with $X$ connected, then $\Sigma(X|D)$ coincides with the cone over the dual intersection complex of $D$. In particular all the cones are smooth, i.e. unimodular. The apex corresponds to the open stratum $X \setminus D$.

If $(X|D)$ is a toric pair then $\Sigma(X|D)$ is isomorphic to the fan of $X$. However, the tropicalisation does not come with a preferred embedding into a vector space; such an embedding depends on the choice of a torus action.

\begin{remark} \label{rmk: defn of sigmaS} The cone $\sigma_S$ associated to a locally-closed stratum $S$ has the following alternative description. Consider the open neighbourhood $U_S \subseteq X$ of $S$ given as the union of locally-closed strata whose closures contain $S$. Then $\sigma_S$ is the cone dual to the monoid of effective divisors in $U_S$ that are supported on $U_S \cap D$. For example, the locally-closed stratum $S=D_i^\circ$ has
\[ U_S = X \setminus \cup_{j \neq i} D_j. \]
The monoid is $\N$ with generator $D_i^\circ$, and so $\sigma_S = \RR_{\geq 0}$. For further details see \cite[Section~1.3]{AbramovichKaru}.
\end{remark}

\subsection{Artin fans} Every cone complex has an algebro-geometric avatar: its Artin fan. This provides a bridge between the geometry of $(X|D)$ and the combinatorics of $\Sigma(X|D)$. For detailed discussions, see \cite{AbramovichWiseBirational,AbramovichChenMarcusWise,AbramovichEtAlSkeletons}.

\begin{definition} An \textbf{Artin cone} is a quotient stack of an affine toric variety by its dense torus
\[ \Acal_\sigma \colonequals [U_\sigma/T_\sigma]. \]
\end{definition}
The assignment $\sigma \mapsto \Acal_\sigma$ is functorial. When the morphism $\sigma_1 \hookrightarrow \sigma_2$ is a face map the associated morphism of Artin cones is an open embedding:
\[ \Acal_{\sigma_1} \hookrightarrow \Acal_{\sigma_2}. \]
For example, the inclusion $\RR_{\geq 0} \hookrightarrow \RR_{\geq 0}^2$ of the first coordinate ray gives the open embedding
\[ [\Aaff^{\!1}/\Gm ] = [(\Aaff^{\!1} \times \Gm)/\Gm^2] \hookrightarrow [\Aaff^{\!2}/\Gm^2].\]

\begin{definition} The \textbf{Artin fan} $\Acal_\Sigma$ of a cone complex $\Sigma$ is the colimit, in the category of algebraic stacks, of the diagram consisting of
\begin{itemize}
\item Artin cones $\Acal_\sigma$ for $\sigma \in \Sigma$,
\item open embeddings $\Acal_{\sigma_1} \hookrightarrow \Acal_{\sigma_2}$ for $\sigma_1 \hookrightarrow \sigma_2$.
\end{itemize}
It is an irreducible stack of pure dimension zero. The Artin fan of a toroidal embedding $(X|D)$ is the Artin fan of its tropicalisation
\[ \Acal_{X|D} \colonequals \Acal_{\Sigma(X|D)}.\]
This can be thought of as a finite topological space, with points corresponding to the strata of $(X|D)$. The topology is the order topology, induced by strata inclusions.
	\end{definition}

\begin{remark}	
The assignment $\Sigma \mapsto \Acal_\Sigma$ gives an equivalence between the $2$-categories of cone complexes (strictly speaking, cone stacks) and Artin fans \cite[Theorem~3]{CavalieriChanUlirschWise}. Happily, the distinction between a cone complex and its Artin fan has grown increasingly blurred.
\end{remark}


The Artin fan $\Acal_{X|D}$ globalises the local toric models for $(X|D)$. Given an open set $U$ and a model $(X|D)|_U \to (U_\sigma | \, \partial U_\sigma)$, the composite morphism
\[ U \rightarrow U_\sigma \to \Acal_\sigma \hookrightarrow \Acal_{X|D} \]
does not depend on the choice of morphism $U \to U_\sigma$. These local morphisms glue to produce a smooth structure morphism 
\begin{equation} p \colon X \to \Acal_{X|D}.\end{equation}
This morphism encodes the toroidal structure on $X$, since $p^{-1}(\partial \Acal_{X|D}) = D$. More generally, given a scheme $X$ and a cone complex $\Sigma$, an arbitrary morphism
\[ p \colon X \to \Acal_\Sigma \] 
defines a \textbf{logarithmic structure} on $X$ via pullback from $\Acal_\Sigma$. The resulting logarithmic scheme is logarithmically smooth (i.e. a toroidal embedding) if and only if $p$ is smooth in the usual sense.

\begin{example} \label{example: snc with connected strata} Let $(X|D=D_1+\ldots+D_k)$ be a simple normal crossings pair and assume that all the strata of $D$ are nonempty and connected. Then $\Sigma(X|D)=\RR_{\geq 0}^k$ and so
\[ \Acal_{X|D} = [\Aaff^{\!k}/\Gm^k]=[\Aaff^{\! 1}/\Gm]^k.\]
 By the definition of the stack quotient, a morphism $X \to [\Aaff^{\! 1}/\Gm]$ consists of a pair $(L,s)$ where $L$ is a line bundle on $X$ and $s \in H^0(X,L)$. In this case, the structure morphism $p \colon X \to \Acal_{X|D}$ encodes the pairs
\[ (\OO_X(D_1),s_{D_1}),\ldots,(\OO_X(D_k),s_{D_k})\]
corresponding to the irreducible components $D_1,\ldots,D_k$.
\end{example}

For a general toroidal embedding, a choice of morphism
\[ \label{eqn: map Artin fan to A1 mod Gm} \Acal_{X|D} \to [\Aaff^{\! 1}/\Gm] \]
determines a pair $(L,s)$ on $X$ via the composite
\[ X \to \Acal_{X|D} \to [\Aaff^{\! 1}/\Gm].\]
Morphisms $\Acal_{X|D} \to [\Aaff^{\!1}/\Gm]$ correspond bijectively to morphisms $\Sigma(X|D) \to \RR_{\geq 0}$ where the target is a cone with lattice $\Z$. Hence this construction generalises the correspondence between piecewise-linear functions on a fan and toric Cartier divisors on the associated toric variety.

\subsection{Isotropic cone complexes} \label{sec: isotropic cone complex} Strata in Artin fans have generically nontrivial isotropy, hence are not Artin fans themselves. To describe these combinatorially, we introduce the language of isotropic cone complexes and isotropic Artin fans. To our knowledge, this is the first treatment of this (relatively straightforward) extension of the theory of cone complexes and Artin fans.

An \textbf{isotropic cone} is a triple
\[ \sigma = (\sigma_{\RR},N_\sigma,N_\sigma^\prime) \]
where $(\sigma_{\RR},N_\sigma)$ is a cone and $N_\sigma^\prime \to N_\sigma$ is a surjection of lattices. Given an isotropic cone we obtain a surjection of tori $T_\sigma^\prime \to T_\sigma$ and an action $T_\sigma^\prime \curvearrowright U_\sigma$. The associated \textbf{isotropic Artin cone} is the stack quotient:
\[ \Bcal_\sigma \colonequals [U_\sigma/T_\sigma^\prime].\]	
A \textbf{map of isotropic cones} $\sigma_1 \to \sigma_2$ consists of commuting lattice maps
\begin{equation}\label{eqn: square map of isotropic cones}
\begin{tikzcd}
N_{\sigma_1}^\prime \ar[r] \ar[d] & N_{\sigma_2}^\prime \ar[d] \\
 N_{\sigma_1} \ar[r] & N_{\sigma_2}	
\end{tikzcd}
\end{equation}
such that the map $N_{\sigma_1} \otimes \RR \to N_{\sigma_2} \otimes \RR$ sends $\sigma_{1\RR}$ into $\sigma_{2\RR}$. A map of isotropic cones induces a morphism of the associated isotropic Artin cones.

A \textbf{face map} is a map of isotropic cones such that the morphism $(\sigma_{1\RR},N_{\sigma_1}) \to (\sigma_{2\RR},N_{\sigma_2})$ is a face map of ordinary cones, and the square \eqref{eqn: square map of isotropic cones} is cartesian; equivalently, the induced map $K_{\sigma_1} \to K_{\sigma_2}$ on kernels is an isomorphism. Face maps induce open embeddings of isotropic Artin cones.

An \textbf{isotropic cone complex} $\Sigma^\prime$ is a diagram of isotropic cones connected by face maps, such that every face of every isotropic cone appears as the image of a face map in the diagram. As before, we prohibit non-identity isomorphisms and the presence of more than one face map between a given pair of isotropic cones. We will only consider isotropic cone complexes which are connected. Given any face map $\sigma_1 \to \sigma_2$, the condition that \eqref{eqn: square map of isotropic cones} is cartesian produces a natural isomorphism
\[ K_{\sigma_1} \colonequals \operatorname{Ker}(N_{\sigma_1^\prime} \to N_{\sigma_1}) \cong \operatorname{Ker}(N_{\sigma_2}^\prime \to N_{\sigma_2}) \equalscolon K_{\sigma_2} . \]
Since $\Sigma^\prime$ is connected, this kernel functions as a canonical \textbf{global isotropy group}. We denote it by
\[ K \colonequals K(\Sigma^\prime).\]
Note that $K = N_0^\prime$ where $0 \in \Sigma^\prime$ is the unique minimal isotropic cone.

Given an isotropic cone complex $\Sigma^\prime$ the associated \textbf{isotropic Artin fan} is the colimit:
\[ \Bcal_{\Sigma^\prime} \colonequals \varinjlim_{\sigma \in \Sigma^\prime} [U_\sigma/T_\sigma^\prime].\] 
This construction is functorial. The unique minimal isotropic cone $0 \in \Sigma^\prime$ corresponds to the dense open stratum $\Bcal T_K \hookrightarrow \Bcal_{\Sigma^\prime}$ where $T_K \colonequals K \otimes \Gm$.

An ordinary cone complex $\Sigma$ can be viewed as an isotropic cone complex by taking $N_{\sigma}^\prime \colonequals N_{\sigma}$ for all $\sigma \in \Sigma$. In this case the isotropic Artin fan coincides with the ordinary Artin fan.

Conversely an isotropic cone complex $\Sigma^\prime$ induces an ordinary cone complex $\Sigma$ by forgetting the lattices $N_\sigma^\prime$. There is a map of isotropic cone complexes $\Sigma^\prime \to \Sigma$ which induces a morphism of isotropic Artin fans
\begin{equation} \label{eqn: map isotropic Artin fan to ordinary Artin fan} \Bcal_{\Sigma^\prime} \to \Acal_\Sigma.\end{equation}

\begin{lemma} \label{lem: isotropic Artin fan as a gerbe} Consider an isotropic cone complex $\Sigma^\prime$ with global isotropy group $K$ and underlying cone complex $\Sigma$. The morphism \eqref{eqn: map isotropic Artin fan to ordinary Artin fan} is a gerbe banded by $T_K \colonequals K \otimes \Gm$.	
\end{lemma}

\begin{proof} Locally on the base, the morphism takes the form
\[ [U_\sigma/T^\prime_\sigma] \to [U_\sigma/T_\sigma]. \]
Moreover, there is a short exact sequence:
\[ 0 \to T_K \to T^\prime_\sigma \to T_\sigma \to 0.\]
The claim then follows from Lemma~\ref{lem: two quotients compared} below.
\end{proof}

\begin{lemma} \label{lem: BH to BG a principal bundle} Consider a short exact sequence of algebraic tori:
\[ 0 \to H \to G \to G/H \to 0. \]
Then the following square, where $\Bcal 0 = \Speck$, is cartesian:
\bcd
\Bcal H \ar[r] \ar[d] \ar[rd,phantom,"\square"] & \Bcal 0 \ar[d] \\
\Bcal G \ar[r] & \Bcal (G/H).
\ecd
In particular, $\Bcal H \to \Bcal G$ is a principal $G/H$-bundle, and $\Bcal G \to \Bcal (G/H)$ is a gerbe banded by $H$.
\end{lemma}

\begin{proof} Fix an arbitrary test scheme $S$ with a morphism $S \to \Bcal G$ corresponding to a principal $G$-bundle $P_G$. The short exact sequence
\[ 0 \to H \to G \to G/H \to 0 \]
induces a long exact cohomology sequence
\begin{equation} \label{eqn: long exact cohomology sequence torsors} \cdots \to H^0(S,G/H) \to H^1(S,H) \to H^1(S,G) \to H^1(S,G/H) \to \cdots \end{equation}
Exactness at $H^1(S,G)$ shows that the principal $G/H$-bundle $P_G/H$ is trivial if and only if $P_G \cong (P_H \times G)/H$ for some principal $H$-bundle $P_H$. This proves that the square is cartesian. Note that exactness at $H^1(S,H)$ shows that the set of lifts $P_H$ of $P_G$ forms a torsor under $G/H$.\end{proof}

\begin{lemma} \label{lem: two quotients compared} Consider a short exact sequence of algebraic tori:
\[ 0 \to H \to G \to G/H \to 0. \]
Given an algebraic stack $X$ and an action $G/H \curvearrowright X$, let $G \curvearrowright X$ be the action induced by the homomorphism $G \to G/H$. Then the following square is cartesian
\bcd
{[X/G]} \ar[r] \ar[d] \ar[rd,phantom,"\square"] & {[X/(G/H)]} \ar[d] \\
\Bcal G \ar[r] & \Bcal(G/H)
\ecd
and the horizontal morphisms are gerbes banded by $H$.
\end{lemma}

\begin{proof}
Fix an arbitrary test scheme $S$. A morphism from $S$ to the fibre product is given by the
following data: a principal $G$-bundle $P_G \to S$, a principal $(G/H)$-bundle $P_{G/H} \to S$, a $(G/H$)-equivariant morphism $P_{G/H} \to X$ and an isomorphism $P_G/H \cong P_{G/H}$. The composition $P_G \to P_{G/H} \to X$ is $G$-equivariant, producing a morphism from $S$ to $[X/G]$. This shows that the square is cartesian. It then follows from the previous lemma that the horizontal morphisms are gerbes banded by $H$.	
\end{proof}

\begin{remark} Lemma~\ref{lem: BH to BG a principal bundle} admits a convenient mnemonic. Writing $A/B$ to denote the fibre of a map $A \to B$, the fact that $\Bcal H \to \Bcal G$ is a principal $G/H$-bundle can be formulated as:
\[ \dfrac{1/H}{1/G} = G/H.\]
Similarly, the fact that $\Bcal G \to \Bcal(G/H)$ is a gerbe banded by $H$ can be formulated as:
\[ \dfrac{1/G}{1/(G/H)} = 1/H. \]
\end{remark}

\subsection{Strata in Artin fans} \label{sec: strata in Artin fans} We now use the theory of isotropic cone complexes to study strata in Artin fans. In the toric context, a cone $\sigma \in \Sigma$ defines three loci in $X_\Sigma$: the affine open $U_\sigma$, the locally-closed orbit $O_\sigma$, and the closed stratum $Z_\sigma$. We discuss the analogues of each of these in the setting of Artin fans.

Fix an ordinary cone complex $\Sigma$ and a cone $\sigma_0 \in \Sigma$. The associated affine open is the Artin cone
\[ \Acal_{\sigma_0} \hookrightarrow \Acal_\Sigma\]
and the associated locally-closed stratum is the composite
\[ \Bcal T_{\sigma_0} \hookrightarrow \Acal_{\sigma_0} \hookrightarrow \Acal_{\Sigma}.\]
The associated closed stratum will be described via an isotropic cone complex. Consider a cone $\sigma \in \Sigma$ such that $\sigma_0 \subseteq \sigma$. We let
\[ \sigma/\sigma_0 \]
denote the image of $\sigma$ under the quotient morphism $N_\sigma \otimes \RR \to (N_\sigma/N_{\sigma_0}) \otimes \RR$. This is strictly convex because $\sigma_0 \subseteq \sigma$ is a face. We then denote and define the \textbf{isotropic star complex} of $\sigma_0 \in \Sigma$:
\[ \Star(\sigma_0, \Sigma) \colonequals \big\{ (\sigma/\sigma_0, N_{\sigma/\sigma_0}, N_\sigma) \mid \sigma_0 \subseteq \sigma \big\}.\]
We write $\sigma=(\sigma/\sigma_0,N_{\sigma/\sigma_0},N_{\sigma})$ for the isotropic cone corresponding to $\sigma$. The isotropic cone complex $\Star(\sigma_0,\Sigma)$ is connected: it has a unique minimal isotropic cone $\sigma_0=(0,0,N_{\sigma_0})$ with generic isotropy $N_{\sigma_0}$. The associated isotropic Artin fan is
\[ \Bcal_{\Star(\sigma_0,\Sigma)} \colonequals \varinjlim [U_{\sigma/\sigma_0}/T_\sigma] \]
where the colimit is over cones $\sigma \in \Sigma$ with $\sigma_0 \subseteq \sigma$. It contains $\Bcal T_{\sigma_0}$ as an open dense stratum.

For each $\sigma \in \Star(\sigma_0,\Sigma)$ there is an embedding $U_{\sigma/\sigma_0} \hookrightarrow U_\sigma$ giving the closed stratum corresponding to the face $\sigma_0 \subseteq \sigma$. This is $T_\sigma$-equivariant and hence descends to a closed embedding $[U_{\sigma/\sigma_0}/T_\sigma] \hookrightarrow \Acal_\sigma$. These glue to produce a closed embedding
\[ \Bcal_{\Star(\sigma_0,\Sigma)} \hookrightarrow \Acal_\Sigma\]
which gives the closed stratum corresponding to $\sigma_0 \in \Sigma$.

As in the previous section, from the isotropic cone complex $\Star(\sigma_0,\Sigma)$ we obtain an ordinary cone complex by forgetting the additional lattices. We refer to this as the \textbf{reduced star complex}:
\[ \Sigma/{\sigma_0} \colonequals \{ (\sigma/\sigma_0, N_{\sigma/\sigma_0}) \mid \sigma \in \Star(\sigma_0,\Sigma)\}.\]
As in Section~\ref{sec: isotropic cone complex}, there is a morphism
\begin{equation} \label{eqn: map isotropic star fan to reduced star fan} \Bcal_{\Star(\sigma_0,\Sigma)} \to \Acal_{\Sigma/\sigma_0} \end{equation}
which is a gerbe banded by $T_{\sigma_0}$. 

\begin{remark} \label{rmk: isotropic Artin fan smooth complexes}
Suppose that all the cones of $\Star(\sigma_0,\Sigma)$ are smooth. Then every lattice $N_{\sigma}$ has a natural basis given by the primitive ray generators. This defines natural splittings $N_{\sigma} = N_{\sigma_0} \times N_{\sigma/\sigma_0}$ which are compatible under restricting to faces. We therefore obtain a natural trivialisation of \eqref{eqn: map isotropic star fan to reduced star fan}: $\Bcal_{\Star(\sigma_0,\Sigma)} \cong \Acal_{\Sigma/\sigma_0} \times \Bcal T_{\sigma_0}$.
\end{remark}

Now let $(X|D)$ be a toroidal embedding with tropicalisation $\Sigma=\Sigma(X|D)$. Let $\sigma \in \Sigma$ be a cone and $X_\sigma \hookrightarrow X$ the corresponding closed stratum. There is a cartesian square
\[
\begin{tikzcd}
X_{\sigma} \ar[r] \ar[d] \ar[rd,phantom,"\square"] & X \ar[d] \\
\Bcal_{\Star(\sigma,\Sigma)} \ar[r] & \Acal_{\Sigma}.	
\end{tikzcd}
\]
The stratum $X_{\sigma}$ is itself a toroidal embedding, with boundary consisting of the intersection with the components of $D$ that do not contain $X_{\sigma}$. The map to its Artin fan is the composite
\[ X_\sigma \to \Bcal_{\Star(\sigma,\Sigma)} \to \Acal_{\Sigma/\sigma}.\]

Now let $\rho \colon Y \to X$ be a morphism of toroidal embeddings. Consider a cone $\sigma_Y \in \Sigma_Y$ and let $\sigma_X \in \Sigma_X$ be the minimal cone containing the image of $\sigma_Y$. Then $\rho$ restricts to a morphism between the corresponding closed strata
\[ Y_{\sigma_Y} \to X_{\sigma_X}.\]
This is a toroidal morphism, i.e. there is a $2$-commuting square
\begin{equation} \label{eqn: commuting square boundary strata}
\begin{tikzcd}
Y_{\sigma_Y} \ar[r] \ar[d] \ar[rd,phantom] & X_{\sigma_X} \ar[d] \\
\Acal_{\Sigma_Y/\sigma_Y} \ar[r] & \Acal_{\Sigma_X/\sigma_X}.	
\end{tikzcd}
\end{equation}

\subsection{Toroidal modifications} \label{sec: toroidal modification}
A central theme in toroidal geometry is that toric constructions on the Artin fan pull back to toroidal constructions on $X$. We discuss an important instance of this phenomenon: toroidal modifications.

Following \cite[Definition 1.3.1]{MR20} we permit subdivisions which are injective but not surjective on supports. This flexibility is useful for applications in enumerative geometry.
\begin{definition} \label{def: subdivision} Let $\Sigma$ be a cone complex. An \textbf{open subdivision} of $\Sigma$ is a morphism of cone complexes
\begin{equation*} \Sigma^\dag \to \Sigma \end{equation*}
such that the induced map on supports $|\Sigma^\dag| \to |\Sigma|$ is injective, and such that the lattice of every cone in $\Sigma^\dag$ is mapped onto a saturated sublattice of the lattice of a cone in $\Sigma$.\end{definition}

\begin{definition} Consider a toroidal embedding $(X|D)$ and let $\Sigma=\Sigma(X|D)$. Given an open subdivision $\Sigma^\dag \to \Sigma$, the associated \textbf{open toroidal modification} is defined as the fibre product
\[
\begin{tikzcd}
X^\dag \ar[r] \ar[d,"p^\dag"] \ar[rd,phantom,"\square" xshift=0.1cm] & X \ar[d,"p"] \\
\Acal_{\Sigma^\dag} \ar[r] & \Acal_\Sigma.	
\end{tikzcd}
\]
The morphism $X^\dag \to X$ is birational and an isomorphism over $X \setminus D$. It is proper and surjective if the open subdivision $\Sigma^\dag \to \Sigma$ is surjective on supports, i.e. if it is a subdivision in the usual sense. We set $D^\dag = (p^\dag)^{-1}(\partial \Acal_{\Sigma^\dag})$ and thus obtain a toroidal morphism
\[ (X^\dag|D^\dag) \to (X|D).\]
\end{definition}
Toroidal (and, more generally, logarithmic) modifications are central to the modern study of intersections on moduli spaces \cite{RanganathanSantosParkerWise1, RanganathanSantosParkerWise2, BattistellaCarocci, DhruvExpansions, HolmesPixtonSchmitt, RanganathanProducts, MaxContacts, MolchoPandharipandeSchmitt, MolchoRanganathan, BNR2,HMPPS}.

\section{Tropical expansions}\label{sec: tropical expansions}
\noindent Let $(X|D)$ be a toroidal embedding with tropicalisation $\Sigma=\Sigma(X|D)$. We consider tropical expansions of $\Sigma$ parametrised by an arbitrary base cone $\tau$. These have been studied by many authors, in various guises and levels of generality \cite{MumfordAbelian,NishinouSiebert,BurgosGilSombra,FosterRanganathan,MR20}.
\begin{definition}\label{def: tropical expansion general} A \textbf{tropical expansion} of $\Sigma$ over $\tau$ consists of an open subdivision of cone complexes
\begin{equation*} \Upsilon = (\Sigma \times \tau)^\dag \to \Sigma \times \tau \end{equation*}
such that the projection $\p \colon \Upsilon \to \tau$ satisfies the following conditions for every cone $\omega \in \Upsilon$
\begin{itemize}
\item $\p$ maps $\omega$ surjectively onto a face of $\tau$ ($\p$ is combinatorially flat);
\item $\p$ maps $N_{\omega}$ surjectively onto the lattice of its image face ($\p$ is combinatorially reduced).
\end{itemize}
\end{definition}

There is a diagram of cone complexes
\begin{equation} \label{eqn: diagram Upsilon mapping to Sigma and tau}
\begin{tikzcd}
\Upsilon \ar[r,"\mathsf{r}"] \ar[d,"\mathsf{p}"] & \Sigma \\
\tau	
\end{tikzcd}
\end{equation}
with $\p$ combinatorially flat and reduced. The open subdivision $\Upsilon \to \Sigma \times \tau$ gives rise to an open toroidal modification $\Xcal_\Upsilon \to X \times U_\tau$ where $U_\tau = \Speck[S_\tau]$ is the corresponding affine toric variety. There is a diagram of schemes
\[
\begin{tikzcd}
\Xcal_\Upsilon \ar[r,"\rho"] \ar[d,"\pi"] & X \\
U_\tau.	
\end{tikzcd}
\]
The conditions on $\p$ ensure that the morphism $\pi$ is flat with reduced fibres \cite[Lemmas 4.1 and 5.2]{AbramovichKaru}. The morphism $\rho$ collapses every fibre of $\pi$ onto the unexpanded target $X$. We refer to this diagram as a \textbf{tropical expansion} of $(X|D)$ over $U_\tau$.

Tropical expansions encompass several familiar constructions, including the Mumford degeneration in toric geometry and the degeneration to the normal cone of a complete intersection. For further examples, see \cite[Section~4]{CN21} and Section~\ref{sec: examples}.

\subsection{Families of polyhedral subdivisions} Fix a tropical expansion $\Upsilon \to \Sigma \times \tau$. For every point $\f \in |\tau|$ we intersect the fibre $\p^{-1}(\f) \subseteq |\Upsilon|$ with the cones of $\Upsilon$ to produce a polyhedral complex
\[ \Upsilon_{\f}.\]
Precisely, this is a diagram of abstract polyhedra connected by face maps, where a polyhedron is the intersection of finitely many closed half-spaces in a real vector space, and an isomorphism of polyhedra is an affine map of vector spaces which identifies the polyhedra. The poset structure is elucidated in Section~\ref{sec: polyhedral conical} below.

This polyhedral complex has support $|\Upsilon_{\f}| = \p^{-1}(\f) \subseteq |\Sigma|$. We view $\p$ as a family of open polyhedral subdivisions of $\Sigma$ parametrised by $\f \in |\tau|$. The \textbf{combinatorial type} of $\Upsilon$ at $\f$ is the data of the polyhedral complex $\Upsilon_{\f}$ together with:
\begin{enumerate}
\item For each polyhedron $P \in \Upsilon_{\f}$ the minimal cone $\sigma_P \in \Sigma$ such that $|P| \subseteq |\sigma_P|$.
\item For each oriented edge $\vec{E} \in \Upsilon_{\f}$ the integral slope $m_{\vec{E}} \in N_{\sigma_E}$.	
\end{enumerate}
This data is constant on the relative interior of each face of $\tau$. The lengths of the bounded edges of the polyhedra in $\Upsilon_\f$ depend on the precise choice of point $\f$ in the base. Specialising $\f$ to a face of $\tau$ collapses certain polyhedra in $\Upsilon_{\f}$, giving rise to a simpler combinatorial type. In the limit $\f=0$ we obtain a cone complex
\[ \Sigma^\dag=\Upsilon_0 \]
called the asymptotic cone complex. The restriction of $\mathsf{r}$ in \eqref{eqn: diagram Upsilon mapping to Sigma and tau} from $\Upsilon$ to $\Upsilon_0$ exhibits $\Sigma^\dag$ as an open (conical) subdivision of $\Sigma$. The associated open toroidal modification $X^\dag \to X$ is the general fibre of $\pi$.

\subsection{Polyhedral-conical dictionary} \label{sec: polyhedral conical}

The poset of cones $\omega \in \Upsilon$ is equal to the poset of  pairs
\[ (\kappa, P)\]
where $\kappa \subseteq \tau$ is a face and $P$ is a polyhedron in $\Upsilon_{\f}$ for $\f \in |\kappa|$ an arbitrary interior point. Under this correspondence we have $\p(\omega)=\kappa$. Moreover,
\[ (\kappa_1,P_1) \leq (\kappa_2,P_2) \]
if and only if $\kappa_1 \subseteq \kappa_2$ and $P_1 \subseteq \widetilde{P}_2$, where $\widetilde{P}_2$ is the limit of $P_2$ under the specialisation from the combinatorial type over $\kappa_2$ to the combinatorial type over $\kappa_1$. We have
\[ \dim \omega = \dim \kappa + \dim P.\]

\subsection{Central fibre}\label{sec: central fibre} The central fibre of $\pi$ over the torus-fixed point $0 \in U_\tau$ will be denoted
\begin{equation*} Y_\Upsilon \colonequals \pi^{-1}(0).\end{equation*}
Strata in $Y_\Upsilon$ are indexed by cones $\omega \in \Upsilon$ with $\p(\omega)=\tau$. By the polyhedral-conical dictionary, the poset of such cones is equal to the poset of polyhedra $P$ in the polyhedral complex $\Upsilon_{\f}$ for $\f \in |\tau|$ an arbitrary interior point. For every such polyhedron $P$ we let $\omega_P \in \Upsilon$ denote the corresponding cone.

The irreducible components of $Y_\Upsilon$ are the maximal strata, hence are indexed by the minimal polyhedra in $\Upsilon_{\f}$: the vertices. Let $V$ denote the set of such vertices and for $v \in V$ let $Y_v$ denote the corresponding irreducible component, so that
\begin{equation*} Y_\Upsilon = \bigcup_{v \in V} Y_v.\end{equation*}
The irreducible component $Y_v$ is the closed stratum of $\Xcal_\Upsilon$ corresponding to the cone
\[ \omega_v \in \Upsilon \]
which is mapped isomorphically onto $\tau$ by $\ptrop$. We also let
\[ \sigma_v \in \Sigma \]
denote the unique minimal cone containing $v$. By definition, this is the minimal cone through which the composition $\omega_v \hookrightarrow \Upsilon \to \Sigma$ factors, where $\Upsilon \to \Sigma$ is the map $\rtrop$ from \eqref{eqn: diagram Upsilon mapping to Sigma and tau}. Let $X_v = X_{\sigma_v} \hookrightarrow X$ denote the corresponding closed stratum. The fibrewise collapsing morphism $\rho \colon \Xcal_{\Upsilon} \to X$ restricts to a collapsing morphism
\[ \rho_v \colon Y_v \to X_v.\]
In simple situations, $\rho_v$ is a toric variety bundle. However, this is not the case in general. It can fail to be flat, or can be flat but with reducible fibres; see Section~\ref{sec: examples}. The principal aim of this paper is to provide a combinatorial recipe to construct $Y_v$ from $X_v$.

\subsection{Tropical position maps} \label{sec: tropical position} For each $v \in V$ there is a linear map
\[ \varphi_v \colon \tau \to \sigma_v\]
recording the position of the vertex $v$ in terms of the tropical parameters. This was introduced in \cite[Definition~1.4]{CN21} in order to describe rubber automorphisms of tropical expansions. Formally it is given by
\[ \varphi_v(\f) = \rtrop(\omega_v \cap \p^{-1}(\f))\]
for $\f \in |\tau|$, where $\rtrop \colon \omega_v \to \sigma_v$ is the restriction of the map in \eqref{eqn: diagram Upsilon mapping to Sigma and tau}.

 The open subdivision $\Upsilon \to \Sigma \times \tau$ restricts to an inclusion $\omega_v \hookrightarrow \sigma_v \times \tau$ corresponding to an inclusion $N_{\omega_v} \hookrightarrow N_{\sigma_v} \times N_\tau$. The following result originally appears in \cite[proof of Theorem~1.5]{CN21}. It is used heavily in Section~\ref{sec: expansion components}.

\begin{proposition} \label{tropical position map gives identification of lattices} The tropical position map $\varphi_v$ produces a natural isomorphism of lattices
\[ (N_{\sigma_v} \times N_\tau)/N_{\omega_v} = N_{\sigma_v}. \]
\end{proposition}

\begin{proof} Consider the map
\begin{equation} \label{eqn: map giving isomorphism of quotient lattices} \operatorname{Id}_{\sigma_v}\oplus(-\varphi_v) \colon N_{\sigma_v} \times N_\tau \to N_{\sigma_v}.\end{equation}
This is surjective, and we will show that its kernel is the image of $N_{\omega_v} \hookrightarrow N_{\sigma_v} \times N_\tau$. The map $\ptrop \colon \Upsilon \to \tau$ restricts to an isomorphism of cones $\omega_v \to \tau$ giving an isomorphism $N_{\omega_v} \cong N_\tau$. Composing with this isomorphism replaces the inclusion $N_{\omega_v} \hookrightarrow N_{\sigma_v} \times N_\tau$ with the inclusion:
\[ (\varphi_v, \operatorname{Id}_\tau) \colon N_\tau \hookrightarrow N_{\sigma_v} \times N_\tau. \]
The claim now follows from the following short exact sequence:
\[ 0 \to N_\tau \xrightarrow{(\varphi_v,\operatorname{Id}_\tau)} N_{\sigma_v} \times N_\tau \xrightarrow{\operatorname{Id}_{\sigma_v} \oplus (-\varphi_v)} N_{\sigma_v} \to 0.\qedhere \]
\end{proof}

\section{Toric variety bundles and fibrewise GIT} \label{sec: toric variety bundles}

\noindent We develop concrete techniques for constructing and manipulating toric variety bundles. These will be employed in Section~\ref{sec: expansion components} to describe the irreducible components of tropical expansions.

\subsection{Toric variety bundles and the mixing construction} Let $X$ be a base scheme and $Z$ a toric variety. We consider the problem of constructing a toric variety bundle $Y \to X$ with fibre $Z$. 

There are several inequivalent definitions of toric variety bundle in the literature, see e.g. \cite{halic2003families, HuLiuYau, hutoric, BrownToric, OhGIT, KavehManon,HKM}. These arise from different choices of subgroup of $\op{Aut}(Z)$ to contain the transition functions. The following definition is restrictive enough to retain desirable toric structures, but flexible enough that it still leads to an interesting theory.

\begin{definition}\label{def:TVB}
A \textbf{toric variety bundle} over $X$ with fibre $Z$ consists of a diagram
\bcd
P \ar[r,hook] \ar[rd] & Y \ar[d] \\
\, & X
\ecd
where $P \to X$ is a principal $T_Z$-bundle, and which Zariski-locally on $X$ restricts to the diagram
\bcd
X \times T_Z \ar[r,hook] \ar[rd] & X \times Z \ar[d]\\
\, & X.
\ecd
\end{definition}

The following construction appears independently in \cite{SankaranUma} (and perhaps elsewhere).

\begin{construction}[Mixing construction] \label{mixing construction} Consider the following input data:
\begin{enumerate}
\item $X$ an arbitrary base scheme.
\item $\Phi$ a fan in a lattice $N$ (the \textbf{fibre fan}).
\item $P$ a principal $T_{N}$-bundle over $X$ (the \textbf{mixing collection}).	
\end{enumerate}
Denote by $Z$ the toric variety corresponding to the fibre fan. Then the \textbf{mixing bundle} is the following scheme over $X$
\[ Y \colonequals (P \times Z)/T_Z\]
where $T_Z=T_N$ is the dense torus of $Z$ acting antidiagonally on $P \times Z$, so that $(tp,z) \sim (p,tz)$.
\end{construction}

\begin{remark} A mixing collection is equivalently given by any of the following pieces of data:
\begin{enumerate}
\item A principal $T_N$-bundle over $X$.
\item A group homomorphism $L \colon M \to \Pic X$ where $M=\Hom(N,\Z)$.
\item A morphism $X \to \Bcal T_N$.	
\end{enumerate}
Often $N$ will come with a preferred basis, in which case a mixing collection is a list of line bundles indexed by this basis.	
\end{remark}

\begin{lemma} \label{lem: toric bundles and mixing} Defintion~\ref{def:TVB} and Construction~\ref{mixing construction} are equivalent: the mixing construction produces a toric variety bundle, and every toric variety bundle arises from the mixing construction.
\end{lemma}
\begin{proof} Start with a toric variety bundle $P \subseteq Y \to X$ and let $\{ U_i \}_{i \in I}$ be a trivialising cover. On each double overlap $U_{ij}=U_i \cap U_j$ there is a commuting diagram of trivialisations
\begin{equation}\label{eqn: diagram of trivialisations}
\begin{tikzcd}
U_{ij} \times T_Z \ar[r,hook] \ar[d,"\varphi_j^{-1}"] & U_{ij} \times Z \ar[d,"\psi_j^{-1}"] \\
P|_{U_{ij}} \ar[r,hook] \ar[d,"\varphi_i"] & Y|_{U_{ij}} \ar[d,"\psi_i"] \\
U_{ij} \times T_Z \ar[r,hook] & U_{ij} \times Z.	
\end{tikzcd}	
\end{equation}
The transition functions $\varphi_i \circ \varphi_j^{-1}$ and $\psi_i \circ \psi_j^{-1}$ correspond equivalently to morphisms
\[ \varphi_{ij} \colon U_{ij} \to T_Z, \qquad \psi_{ij} \colon U_{ij} \to \operatorname{Aut}(Z).\]
To show that $Y \to X$ arises from the mixing construction, we will show that $\psi_{ij}$ factors through the subgroup $T_Z \leq \Aut (Z)$. The action $T_Z \acts Z$ defines a unique extension of $\varphi_{ij}$ to an automorphism of $Z$. By \eqref{eqn: diagram of trivialisations}, $\varphi_{ij}$ and $\psi_{ij}$ coincide on the dense open $T_Z$ and hence on all of $Z$. Therefore $\psi_{ij}$ factors through $T_Z$ as claimed.

Conversely, let $Y = (P \times Z)/T_Z \to X$ be obtained via the mixing construction. We have
\[ P = (P \times T_Z)/T_Z \]
and so clearly there is an inclusion $P \hookrightarrow Y$ which locally restricts to $X \times T_Z \hookrightarrow X \times Z$.
\end{proof}
	
\subsection{Universal toric variety bundle} \label{sec: universal toric variety bundle} Consider the Artin fan $\Acal_\Phi$. There is a natural isomorphism
\begin{equation} \label{eqn: Artin fan of toric variety} \Acal_\Phi = [Z/T_N] \end{equation}
since both arise as the colimit of quotients $\Acal_\varphi = [U_\varphi/T_\varphi]$ for $\varphi \in \Phi$. For $[Z/T_N]$ this relies on the following fact: given $\varphi \in \Phi$ a cone with lattice $N_{\varphi}$, we can use the lattice inclusion $N_\varphi \subseteq N$ to view $\varphi$ as a cone in $N$. There is a map of cones $(\varphi_{\R},N_\varphi) \to (\varphi_{\R},N)$ which induces the following closed embedding:
\[ U_{\varphi,N_\varphi} \hookrightarrow U_{\varphi,N_\varphi} \times (N/N_\varphi \otimes \Gm) \cong U_{\varphi,N}. \]
This closed embedding becomes an isomorphism after quotienting by dense tori
\[ [U_{\varphi,N_\varphi}/T_{N_\varphi}] = [U_{\varphi,N}/T_{N}]\]
because on the right-hand side the larger torus $T_{N}$ cancels out the additional torus factors in $U_{\varphi,N}$ (in short: Artin cones do not countenance enlargements of the lattice). We can thus identify $\Acal_\varphi$ with the quotient $[U_{\varphi,N}/T_N]$ where $U_{\varphi,N} \hookrightarrow Z$ is the associated affine toric open. This globalises to the identification \eqref{eqn: Artin fan of toric variety}. From this we obtain a map:
\begin{equation} \label{eqn: universal TVB} \Acal_\Phi \to \Bcal T_N. \end{equation}
In the language of Section~\ref{sec: isotropic cone complex}, this arises from the map $\Phi \to N$ of isotropic cone complexes. The following lemma shows that \eqref{eqn: universal TVB} is the \textbf{universal toric variety bundle} with fibre fan $\Phi$.

\begin{lemma} \label{lem: universal toric variety bundle} Fix a fibre fan $\Phi$ and a mixing collection $X \to \Bcal T_N$. Then the associated toric variety bundle $Y \to X$ arises as a fibre product:
\[ 
\begin{tikzcd}
Y \ar[r] \ar[d] \ar[rd,phantom,"\square"] & \Acal_\Phi \ar[d] \\
X \ar[r] & \Bcal T_N.	
\end{tikzcd}
\]
Thus, any pullback of $\Acal_\Phi \to \Bcal T_N$ is a toric variety bundle, and all toric variety bundles with fibre fan $\Phi$ arise in this way.
\end{lemma}

\begin{proof} Temporarily let $F$ denote the fibre product:
\[ F \colonequals X \times_{\Bcal T_N} \Acal_\Phi.\]
We will show $F = Y \colonequals (P \times Z)/T_N$. Given a test scheme $S$, a morphism $S \to F$ consists of: a morphism $g \colon S \to X$, a principal $T_N$-bundle $Q \to S$ with an equivariant morphism $Q \to Z$, and an isomorphism $Q \cong g^\star P$ where $P$ is the principal bundle on $X$ induced by $X \to \Bcal T_N$. This data is equivalent to: a morphism $g \colon S \to X$, a principal $T_N$-bundle $Q \to S$ and an equivariant morphism
\[ Q \to g^\star P \times Z \]
where $T_N$ acts antidiagonally on the target. This shows $Y= (P \times Z)/T_N$ as claimed.
\end{proof}

\subsection{Identifying the mixing collection} \label{sec: identifying mixing collection} Consider a toric variety bundle $P \subseteq Y \to X$, which by Lemma~\ref{lem: toric bundles and mixing} we may express as a quotient
\[ Y = (P \times Z)/T_N \]
with $T_N \subseteq Z$ the fibre toric variety. Since $T_N \curvearrowright Z$ preserves the toric strata, there are well-defined fibrewise toric strata in $Y$ which map surjectively onto $X$. This is one of the major advantages of Definition~\ref{def:TVB}, and will play a crucial role in the discussion of tropical expansions.

Let $\Phi$ be the fibre fan corresponding to $Z$ and let $\Phi(1)$ denote the set of rays. For $\rho \in \Phi(1)$ we let $D_\rho \subseteq Z$ denote the corresponding toric Weil divisor and
\[ \Dcal_\rho = (P \times D_\rho)/T_N \subseteq (P \times Z)/T_N = Y \]
the horizontal divisor in the toric variety bundle. Let $M$ be the character lattice of $T_N$. Recall that for each $m \in M$ there is a relation
\[ \sum_{\rho \in \Phi(1)} \langle m, v_\rho \rangle D_\rho = 0 \]
in $\Cl(Z)$, where $v_\rho \in N$ is the primitive lattice generator of the ray. In the bundle $Y$, the analogous relations hold only up to line bundles pulled back from the base. The following lemma shows that these line bundles are equivalent to the data of the mixing collection.

\begin{lemma} \label{lem: identifying mixing collection} Let $\pi \colon Y \to X$ be a toric variety bundle obtained via the mixing construction (Construction~\ref{mixing construction}), with fibre fan $\Phi$ and mixing collection encoded by a homomorphism
\[ L \colon M \to \Pic X.\]
Then for each $m \in M$ we have the following relation in $\Cl Y$
\begin{equation} \label{eqn: divisor relation in toric bundle} \sum_{\rho \in \Phi(1)} \langle m,v_\rho \rangle \Dcal_\rho = \pi^\star L(m).\end{equation}
\end{lemma}

\begin{proof} Note that $M=\Cl_{T_N}(\pt) = \Cl(\Bcal T_N)$. Consider the cartesian diagram from Lemma~\ref{lem: universal toric variety bundle}:
\[
\begin{tikzcd}
Y \ar[d,"\pi"] \ar[r] \ar[rd,phantom,"\square"] & {[Z/T_N]} \ar[d] \\
X \ar[r] & \Bcal T_N.
\end{tikzcd}
\]
Fix $m \in M=\Cl(\Bcal T_N)$. Pulling back along the two routes gives the two sides of \eqref{eqn: divisor relation in toric bundle}.
\end{proof}

\subsection{Fibrewise GIT} The results of this section are not used elsewhere in the paper. They are included as they provide concrete techniques for manipulating toric variety bundles.

\subsubsection{GIT construction} For this section we assume that the fibre fan $\Phi$ is simplicial and not contained in any proper linear subspace; equivalently, the associated toric variety $Z$ is $\Q$-factorial and contains no torus factors. In this context, it is well-known that $Z$ arises as a GIT quotient
\[ Z = \Aaff^{\!\Phi(1)} \sslash_\theta G \]
where $G=\Hom(\Cl Z, \Gm)$ is a finite extension of an algebraic torus and $\Phi(1)$ is the set of rays \cite{CoxHomogeneous} (see also \cite[Chapters 5 and 14]{CLS}). The exact sequence
\begin{equation} \label{eqn: short exact sequence class group} 0 \to M \to \ZZ^{\Phi(1)} \to \Cl Z \to 0 \end{equation}
dualises, since $\Gm$ is divisible, to an exact sequence
\begin{equation} \label{eqn: GIT exact sequence tori} 0 \to G \to \Gm^{\Phi(1)} \to T_Z \to 0. \end{equation}
The inclusion $G \hookrightarrow \Gm^{\Phi(1)}$ defines an action $G \curvearrowright \Aaff^{\!\Phi(1)}$ of which $Z$ is the GIT quotient. The unstable locus
\[ B(\Phi) \subseteq \Aaff^{\!\Phi(1)}\]
is a union of coordinate subspaces, corresponding to collections of toric divisors in $Z$ with empty intersection. It is induced by a character $\theta$ of $G$ which we fix once and for all (see e.g. \cite[Section~4.2]{CoatesIritaniJiang} for more details).

The GIT construction extends to the relative setting \emph{mutatis mutandis}, as we now explain. This has already appeared at various points in the literature, see e.g. \cite{halic2003families, hutoric,BrownToric, OhGIT}. 

\begin{construction}[Fibrewise GIT] \label{construction: fibrewise GIT} Fix a fan $\Phi$ whose associated toric variety $Z$ has a GIT quotient presentation, as above. Let $X$ be an arbitrary base scheme and choose a homomorphism
\[ K \colon \ZZ^{\Phi(1)} \to \Pic X.\]
The inclusion $G \hookrightarrow \Gm^{\Phi(1)}$ produces an action of $G$ on the total space of the vector bundle
\[ E \colonequals \bigoplus_{\rho \in \Phi(1)} K(e_\rho).\]
The character $\theta$ of $G$ produces a lifted action on the total space of the trivial line bundle over $E$
\begin{align*} G & \curvearrowright  E \times \Aaff^{\! 1}\\
g(y,z) & = (g\cdot y, \theta(g)\cdot z).
\end{align*}
We let $\OO_E(\theta)$ denote this equivariant line bundle. The \textbf{fibrewise GIT quotient} is then defined as 
\[ E \sslash_{\theta} G \colonequals \Proj_{\OO_X} \bigoplus_{k \geq 0} (\pi_\star \OO_E(k \theta))^G\]
where $\pi \colon E \to X$. Since the character $\theta$ is fixed throughout, we suppress it from the notation. 
\end{construction}

The following Theorems~\ref{thm: GIT implies TVB} and \ref{thm: TVB implies GIT} together establish an equivalence between the mixing construction (Construction~\ref{mixing construction}) and the fibrewise GIT construction (Construction~\ref{construction: fibrewise GIT}).
\begin{theorem} \label{thm: GIT implies TVB} Every fibrewise GIT quotient is a toric variety bundle in the sense of Definition~\ref{def:TVB}.
\end{theorem}

\begin{proof} Fix a fibrewise GIT quotient $Y=E \sslash G$ as in Construction~\ref{construction: fibrewise GIT}. Let $E^\circ \subseteq E$ be the complement of all the coordinate subbundles. We claim that
\[ P \colonequals E^\circ \sslash  G \hookrightarrow E \sslash  G = Y\]
is the inclusion of a principal $T_Z$-bundle which locally restricts to $X \times T_Z \hookrightarrow X \times Z$. Over a trivialising open set $U \subseteq X$ for $E$ we have
\[ E^\circ|_U \cong U \times \Gm^{\Phi(1)} \subseteq U \times \Aaff^{\!\Phi(1)} \cong E|_U.\]
Note that $E^\circ$ is disjoint from the unstable locus in $E$, and moreover the action $G \acts E$ restricts to a free action $G \acts E^\circ$. Hence we obtain
\[  (E^\circ \sslash  G)|_U = U \times (\Gm^{\Phi(1)} / G) = U \times T_Z\]
where the last equality holds by \eqref{eqn: GIT exact sequence tori}. This produces the desired inclusion $P\colonequals E^\circ \sslash  G  \hookrightarrow Y$ which clearly restricts to $U \times T_Z \hookrightarrow U \times Z$ locally, as required. We note that $P$ is the principal $T_Z$-bundle corresponding to the composite
\[  M \hookrightarrow \ZZ^{\Phi(1)} \xrightarrow{K} \Pic X. \qedhere\] \end{proof}

\begin{theorem} \label{thm: TVB implies GIT} Consider a toric variety bundle $Y \to X$ with mixing collection encoded in a homomorphism $L \colon M \to \Pic X$. Choose a homomorphism
\[ K \colon \Z^{\Phi(1)} \to \Pic X \]
which restricts to $L$ under the natural inclusion $M \hookrightarrow \ZZ^{\Phi(1)}$. Then $Y$ is equal to the fibrewise GIT quotient associated to $K$.

In particular, if $K_1,K_2 \colon \Z^{\Phi(1)} \to \Pic X$ restrict to the same homomorphism $L \colon M \to \Pic X$ then the fibrewise GIT quotients associated to $K_1$ and $K_2$ are isomorphic.
\end{theorem}
\begin{proof} Let $E=\oplus_{\rho \in \Phi(1)} K(e_\rho)$. On a trivialising open set we have $E|_U = U \times \Aaff^{\!\Phi(1)}$ and the fibrewise GIT construction gives
\[ (E \sslash  G)|_U = U \times Z. \]
The transition functions for $E$ take values in $\Gm^{\Phi(1)}$ and hence the transition functions for $E \sslash  G$ take values in 
\[ \Gm^{\Phi(1)}/G = T_Z.\]
Restricting to the principal bundle $P \subseteq Y$, the fact that $K$ restricts to $L$ implies that the transition functions for $P$ and for $E \sslash  G$ coincide. Hence $Y=E \sslash  G$ as claimed.
\end{proof}

\begin{remark} The fact that the fibrewise GIT quotient only depends on the restriction $K|_M$ generalises the well-known fact that
\[ \PP_X(E \otimes A) \cong \PP_X(E) \]
for $E$ a vector bundle and $A$ a line bundle.	
\end{remark}

\subsubsection{Fibrewise line bundles and homogeneous coordinates} For this section, we assume that the fibre fan $\Phi$ is smooth and not contained in any proper linear subspace; equivalently, the associated toric variety $Z$ is smooth and contains no torus factors. We show that the choice of extension
\[ K \colon \Z^{\Phi(1)} \to \Pic X \]
of the mixing collection not only endows $Y$ with the structure of a GIT quotient, but equips it with special line bundles and relative homogeneous coordinates. These generalise the fibrewise $\OO(1)$ and relative homogeneous coordinates for projective bundles.

\begin{construction}[Fibrewise bundles] Fix input data $(X,\Phi,L)$ as in Construction~\ref{mixing construction} and choose an extension $K$ of $L$ as in the fibrewise GIT construction (Construction~\ref{construction: fibrewise GIT}). Consider the variety:
\begin{equation} \label{eqn: bundle V GIT construction} V \colonequals \bigoplus_{\rho \in \Phi(1)} K(e_\rho) \setminus B(\Phi). \end{equation}
Since $\Phi$ is smooth, $G$ acts freely on $V$ and we have $Y = [V/G]$. We thus obtain a morphism:
\[ Y \to \Bcal G. \]
Moreover, again since $\Phi$ is smooth, $G$ is an algebraic torus with character lattice equal to $\Cl Z = \Pic Z$. The above morphism is thus equivalent to the data of a homomorphism:
\[ \OO_K \colon \Pic Z \to \Pic Y. \]
We refer to the $\OO_K(A)$ for $A \in \Pic Z$ as the \textbf{fibrewise bundles} associated to the extension $K$ of $L$. 
\end{construction}

\begin{lemma} Consider two homomorphisms $K_1,K_2 \colon \Z^{\Phi(1)} \to \Pic X$ which restrict to the same mixing collection $L \colon M \to \Pic X$, and thus produce the same toric variety bundle $Y$ by Theorem~\ref{thm: TVB implies GIT}.

The short exact sequence \eqref{eqn: short exact sequence class group} implies that $K_1 \otimes K_2^{-1}$ descends uniquely to a homomorphism:
\[ R \colon \Pic Z \to \Pic X. \]
Then for $A \in \Pic Z$ we have:
\[ \OO_{K_1}(A) \cong \OO_{K_2}(A) \otimes p^\star R(A).\]	
\end{lemma}

\begin{proof} For $i \in \{1,2\}$ the morphism $\OO_{K_i}$ is equivalent to the data of the principal $G$-bundle
\[ V_i \to Y = [V_i/G] \]
where the prequotient $V_i$ is given as in \eqref{eqn: bundle V GIT construction} by:
\[ V_i \colonequals \bigoplus_{\rho \in \Phi(1)} K_i(e_\rho) \setminus B(\Phi). \]
It is equivalent to establish an isomorphism $V_1 \cong V_2 \otimes p^\star R$ of principal $G$-bundles over $Y$.

Each $V_i$ is itself a toric variety bundle, as we now describe. Consider the subfan $\smash{\widetilde{\Phi}}$ of the positive orthant in $\smash{\Z^{\Phi(1)}}$ consisting of faces whose associated set of rays is contained in a cone of $\Phi$. Then $V_i$ is the toric variety bundle over $X$ with fibre fan $\smash{\widetilde{\Phi}}$ and mixing collection $K_i$.

We are given an isomorphism of principal $\smash{\Gm^{\Phi(1)}}$-bundles over $X$
\[ K_1 \cong K_2 \otimes R \]
where $R$ is viewed as a principal $\Gm^{\Phi(1)}$-bundle using the inclusion $G \to \Gm^{\Phi(1)}$. This produces an isomorphism between the associated toric variety bundles with fibre fan $\smash{\widetilde{\Phi}}$, and the isomorphism commutes with the projections to $Y$ because $K_1/G \cong L \cong K_2/G$, and $R/G$ is trivial.

Denote by $V_2^\prime$ the toric variety bundle associated to $(\widetilde{\Phi},K_2 \otimes R)$, viewed as a principal $G$-bundle over $Y$. By the above we have $V_1 \cong V_2^\prime$; we now compare $V_2^\prime$ to $V_2$. Fix an open cover of $X$ which trivialises $K_2$ and $R$ (and thus $Y$). This pulls back to an open cover of $Y$ on which the map $V_2^\prime \to Y$ takes the form:
\[ U \times Z_{\widetilde{\Phi}} \to U \times Z_{\Phi}.\]
Examining the transition functions, we obtain an isomorphism of principal $G$-bundles over $Y$:
\[ V_2^\prime \cong V_2 \otimes p^\star R. \]
Since $V_1 \cong V_2^\prime$ this completes the proof.
\end{proof}

Finally, we show that in the smooth case the fibrewise GIT quotient admits homogeneous coordinates over the base, as in \cite{CoxFunctor}.
\begin{lemma} Fix a base scheme $X$, a fibre fan $\Phi$, and a homomorphism
\[ K \colon \Z^{\Phi(1)} \to \Pic X\]
with restriction $L \colon M \to \Pic X$. Let $Y = E \sslash G$ denote the fibrewise GIT quotient of Construction~\ref{construction: fibrewise GIT}. Suppose that $\Phi$ is smooth and let $S$ be an arbitrary test scheme. Then a morphism $S \to Y$ is equivalent to the following data:
\begin{enumerate}
	\item a morphism $f \colon S \to X$;
	\item for each $\rho \in \Phi(1)$ a line bundle and section $(A_\rho,u_\rho)$ on $S$;
	\item an isomorphism
	\[ \bigotimes_{\rho \in \Phi(1)} A_\rho^{\otimes \langle m,v_\rho \rangle} \cong f^\star L(m)\] 
	for each $m \in M$, compatible with the group structure on $M$.
\end{enumerate}
The tuple of sections $(u_\rho)_{\rho \in \Phi(1)}$ must avoid the unstable locus $B(\Phi) \subseteq \bigoplus_{\rho \in \Phi(1)} A_\rho$.	
\end{lemma}

\begin{proof} As in the proof of Lemma~\ref{lem: identifying mixing collection} there is a cartesian square
\begin{equation}
\begin{tikzcd} \label{eqn: fibre product functor of points GIT}
Y \ar[r] \ar[d] \ar[rd,phantom,"\square"] & {[Z/T_Z]} \ar[d] \\
X \ar[r] & \Bcal T_Z.
\end{tikzcd}
\end{equation}
Since $\Phi$ is smooth, the GIT quotient coincides with the stack quotient:
\begin{equation} \label{eqn: GIT quotient is stack quotient} Z = [(\Aaff^{\Phi(1)} \setminus B(\Phi)) / G]. \end{equation}
Together with \eqref{eqn: GIT exact sequence tori}, this gives
\[ [Z/T_Z] = [(\Aaff^{\Phi(1)} \setminus B(\Phi)) / \Gm^{\Phi(1)}].\]
The result now follows directly from \eqref{eqn: fibre product functor of points GIT}. Note that the morphism $X \to \Bcal T_Z$ precisely encodes the mixing collection $L$.
\end{proof}

\section{Expansion components}\label{sec: expansion components}

\subsection{Setup} Fix a toroidal embedding $(X|D)$ with tropicalisation $\Sigma=\Sigma(X|D)$. Let $\tau$ be a cone and consider a combinatorial tropical expansion
\bcd
\Upsilon \ar[r,"\rtrop"] \ar[d,"\p"] & \Sigma \\
\tau 
\ecd 
with associated geometric tropical expansion
\bcd
\Xcal_\Upsilon \ar[r,"\rho"] \ar[d,"\pi"] & X \\
U_\tau
\ecd
where $U_\tau = \Spec \kfield[S_\tau]$ is the corresponding affine toric variety. Let $0 \in U_\tau$ be the torus-fixed point. As discussed in Section~\ref{sec: central fibre}, the irreducible components of the central fibre $Y_\Upsilon=\pi^{-1}(0)$ are indexed by vertices $v$ in the corresponding polyhedral subdivision of $\Sigma$,
\[ Y_\Upsilon = \bigcup_{v \in V} Y_v.\]
By the polyhedral-conical dictionary, a vertex $v$ corresponds to a cone $\omega_v \in \Upsilon$ mapped isomorphically onto $\tau$ by $\p$, and $Y_v \hookrightarrow \Xcal_\Upsilon$ is the closed stratum corresponding to $\omega_v$.

Let $\sigma_v \in \Sigma$ be the minimal cone containing $\rtrop(\omega_v)$ and let $X_v \hookrightarrow X$ be the closed stratum corresponding to $\sigma_v$. The morphism $\rho$ restricts to a collapsing morphism
\[ \rho_v \colon Y_v \to X_v. \]
Since the pair $(X|D)$ is arbitrary the stratum $X_v$ is equally arbitrary, and nothing interesting can be said about it. Instead, our goal is to describe the relative geometry of the morphism $\rho_v$. As we will see, this is controlled entirely by the polyhedral combinatorics.

\subsection{Cautionary examples} \label{sec: examples} The naive expectation is that $\rho_v$ should be a toric variety bundle. However, this is false: it may have reducible fibres (Example~\ref{example: flat but reducible fibre}) or even fail to be flat (Example~\ref{example: not flat}). This precludes a simple description of $\rho_v$. In the following sections we study a sequence of open subschemes
\[ Y_v^\circ \hookrightarrow Y_v^\bullet \hookrightarrow Y_v \]
and proceed to describe the relative geometry of each over the base. For $Y_v^\circ$ and $Y_v^\bullet$ we obtain complete descriptions (Sections~\ref{sec: Yvcirc} and \ref{sec: Yvbullet}). For $Y_v$ we obtain a combinatorial criterion for $\rho_v$ to be a toric variety bundle; when this criterion is satisfied, we give a combinatorial recipe to construct $Y_v$ from $X_v$ (Section~\ref{sec: Yv closed stratum}). If the stratum $X_v \hookrightarrow X$ is minimal then $Y_v^\bullet=Y_v$ and so the results of Section~\ref{sec: Yvbullet} suffice.

\begin{example}\label{example: flat but reducible fibre}
Take $X$ a smooth variety and $D=D_1+D_2$ the union of two smooth divisors with nonempty connected intersection. We have $\Sigma=\Sigma(X|D)=\RR^2_{\geq 0}$. Consider the following tropical expansion over the base cone $\tau=\RR_{\geq 0}$ with coordinate $e$.

\begin{figure}[H]
	\centering
		\centering
		\begin{tikzpicture}[scale=.7]
		\tikzstyle{every node}=[font=\normalsize]
		\tikzset{arrow/.style={latex-latex}}
		
		\tikzset{cross/.style={cross out, draw, thick,
         minimum size=2*(#1-\pgflinewidth), 
         inner sep=1.2pt, outer sep=1.2pt}}

\coordinate (O) at (0,0,0);
	 \coordinate (B) at (6,0,0);
	  \coordinate (Bhalf) at (3,0,0);
   \coordinate (C) at (3,5,0);
           \coordinate (v1) at (1.55,2.6,0);
        		
\coordinate (OP) at (0,0,0);
\coordinate (BP) at (4,0,0);
   \coordinate (CP) at (0,6,0);
 \coordinate (v1P) at (0,3,0);
    \coordinate (vl1) at (4,3,0);
      \coordinate (vl3) at (4,6,0);
          
\draw (O)--(B);
\draw (B)--(C);
\draw[blue] (Bhalf)--(v1);
\draw (O)--(C);
\draw[blue] (v1)--(B);

\node at (O) [below] {\small{$\ell_2$}};	
     \node at (B) [below] {\small{$\ell_1$}};
     \node at (C) [above] {\small{$e$}};
     
     \node at (C) [right] {\small{$v_0$}};	
     \node at (v1) [right,blue] {\small{$v_1$}};
  
\foreach \x in {O,B, Bhalf,C}
   \fill (\x) circle (3pt);
   \foreach \x in {v1}
   \fill[blue] (\x) circle (3pt);

\begin{scope}[every coordinate/.style={shift={(10,0,0)}}]
\draw[->] ([c]OP)--([c]BP);
\draw [->]([c]OP)--([c]CP);

\draw[blue,->] ([c]v1P)--([c]vl1);
\draw[blue,->] ([c]v1P)--([c]vl3);
       
       \node at ([c]OP) [below] {\small{$v_0$}};	
     \node at ([c]v1P) [left,blue] {\small{$v_1$}};
     
      \node at ([c]BP) [right] {\small{$\ell_1$}};
          \node at ([c]CP) [above] {\small{$\ell_2$}};

\draw [decorate,decoration={brace,amplitude=5pt,mirror},xshift=0.4pt,yshift=-0.4pt] ([c]v1P)--([c]OP) node[black,midway,xshift=-0.4cm,yshift=0 cm] {\small{$e$}};

   \fill[black] ([c]OP) circle (3pt);
       \fill[blue] ([c]v1P) circle (3pt);

\end{scope}
\end{tikzpicture}
\caption{$Y_{v_1} \to X_{v_1}$ has reducible fibres.}
\label{fig:flatbutnotbundle} 
\end{figure}
\noindent On the left is the height-$1$ slice of the conical subdivision of $\Sigma \times \tau = \RR^3_{\geq 0}$; on the right is the polyhedral subdivision of $\Sigma$ with parameter $e$. These subdivisions are bijective on supports.

The collapsing morphism $Y_{v_1} \to X_{v_1}=D_2$ is a flat family of rational curves, but is not a toric variety bundle: the general fibre is smooth but the fibre over $D_1 \cap D_2$ is nodal with two smooth components. Explicitly, $Y_{v_1}$ is the blowup of the $\PP^1$-bundle $\mathbb P_{D_2}(\mathcal O(D_2) \oplus\mathcal O)$, at the intersection of the infinity section with the fibre over $D_1 \cap D_2$.
\end{example}

\begin{example}\label{example: not flat}
Take $X=\Aaff^{\!3}$ with $D=D_1+D_2+D_3$ the three coordinate planes. Then $\Sigma=\RR^3_{\geq 0}$ with coordinates $\ell_1,\ell_2,\ell_3$. Consider the tropical expansion illustrated in Figure~\ref{fig:notflat}, with base cone $\tau=\RR_{\geq 0}$ with coordinate $e$. This is an open subdivision $\Upsilon \to \Sigma \times \tau$: in the polyhedral picture on the right, it consists only of the black coordinate lines and the blue rays, while in the height-$1$ slice on the left, it consists only of the $1$-skeleton. In particular, $\Upsilon$ only contains cones of dimension $2$ or less. Note that the image of the fibrewise collapsing morphism $\Upsilon \to \Sigma$ intersects the interior of the maximal cone of $\Sigma=\RR_{\geq 0}^3$.

\begin{figure}[h]
	\centering
		\centering
		\begin{tikzpicture} [scale=.6]
		\tikzstyle{every node}=[font=\normalsize]
		\tikzset{arrow/.style={latex-latex}}
		
		\tikzset{cross/.style={cross out, draw, thick,
         minimum size=2*(#1-\pgflinewidth), 
         inner sep=1.2pt, outer sep=1.2pt}}

\coordinate (O) at (0,0,-2);
 \coordinate (B) at (6,0,0);
   \coordinate (C) at (2,6,0);
      \coordinate (A) at (1,0,8);
      \coordinate (v1) at (0.5,0,3);
         \coordinate(L) at (4.3,2.6,0);

\coordinate (OP) at (0,0,0);
\coordinate (BP) at (6,0,0);
   \coordinate (CP) at (0,6,0);
      \coordinate (AP) at (0,0,8);

 \coordinate (v1P) at (0,0,5);
       \coordinate (vl1) at (6,0,5);
      \coordinate (vl2) at (0,6,5);
    \coordinate (vl3) at (5,6,5);

\begin{scope}[every coordinate/.style={shift={(-1,0,0)}}]
          
\draw ([c]O)--([c]B);
\draw ([c]B)--([c]C);
\draw ([c]O)--([c]C);
\draw([c]A)--([c]O);
\draw([c]A)--([c]B);
\draw([c]A)--([c]C);

\draw[dashed,blue] ([c]v1)--([c]B);
\draw[dashed,blue] ([c]v1)--([c]C);
\draw[dashed,blue] ([c]v1)--([c]L);
\draw[dashed,blue] ([c]L)--([c]O);
\draw[dashed,blue] ([c]L)--([c]A);

\node at ([c]O) [below] {\small{$\ell_1$}};	
     \node at ([c]B) [below] {\small{$\ell_2$}};
     \node at ([c]C) [above] {\small{$\ell_3$}};
          \node at ([c]A) [left] {{$e$}};
     \node at ([c]L) [right] {\small{$w$}};

     \node at ([c]v1) [below,blue] {\small{$v_1$}};

\foreach \x in {O,B, C,L,A}
   \fill ([c]\x) circle (3pt);
   \foreach \x in {v1}
   \fill [blue]([c]\x) circle (3pt);

\end{scope}

\begin{scope}[every coordinate/.style={shift={(10,0,0)}}]

\draw [decorate,decoration={brace,amplitude=5pt,mirror},xshift=0.4pt,yshift=-0.4pt] ([c]v1P)--([c]OP) node[black,midway,xshift=0.25cm,yshift=-.2 cm] {\small{$e$}};

\draw[->] ([c]OP)--([c]BP);
\draw [->]([c]OP)--([c]CP);
\draw [->]([c]OP)--([c]AP);

\draw[blue,->] ([c]v1P)--([c]vl3);
\draw[blue,->] ([c]v1P)--([c]vl1);
\draw[blue,->] ([c]v1P)--([c]vl2);

   \fill ([c]OP) circle (3pt);
       \fill [blue]([c]v1P) circle (3pt);

       \node at ([c]OP) [xshift=0.2cm,yshift=0.15cm] {\small{$v_0$}};	
     \node at ([c]v1P) [xshift=-0.25cm,yshift=0.05cm,blue] {\small{$v_1$}};
     
     
      \node at ([c]BP) [right] {\small{$\ell_2$}};
          \node at ([c]CP) [above] {\small{$\ell_3$}};
           \node at ([c]AP) [right] {\small{$\ell_1$}};
           \node at ([c]vl3) [above,blue] {{$\subalign{\ell_1 & =e\\ \ell_2 & =\ell_3}$}};

\end{scope}

\end{tikzpicture}
\caption{$Y_{v_1}\to X_{v_1}$ is not flat.}
\label{fig:notflat} 
\end{figure}

We have $X_{v_1}=D_1$ and the morphism $Y_{v_1} \to D_1$ is toric, with fan map
\begin{equation} \label{eqn: second example toric map} \Sigma_{Y_{v_1}} \to \Sigma_{D_1}\end{equation}
given by the lattice morphism $N_{\Upsilon}/N_{\omega_{v_1}} \to N_{\Sigma}/N_{\sigma_{v_1}}$. The ray
\[ (\RR_{\geq 0} v_1+ \RR_{\geq 0}w)/(N_{\omega_{v_1}}\otimes \R) \]
 in $\Sigma_{Y_{v_1}}$ corresponds to the ray in the polyhedral subdivision given by $\{ \ell_1=e, \ell_2 = \ell_3\}$. Under \eqref{eqn: second example toric map} it is mapped onto the diagonal of $\Sigma_{D_1}=\RR^2_{\geq 0}$. By \cite[Lemma~4.1]{AbramovichKaru} we conclude that $Y_{v_1} \to D_1$ is not flat. Geometrically: the fibre over a general point of $D_1$ is $1$-dimensional, but the fibre over the point $D_1 \cap D_2 \cap D_3$ is $2$-dimensional. Combinatorial semistable reduction \cite{AbramovichKaru,MolchoSS,ALT-Semistable} can be used to flatten the morphism, by subdividing $\Sigma_{D_1}$ along the diagonal. This corresponds to blowing up $D_1$ at the point $D_1 \cap D_2 \cap D_3$.
\end{example}

\subsection{Standard mixing collection} \label{sec: standard mixing collection} The locally-closed stratum $X_v^\circ \hookrightarrow X$ is obtained from the closed stratum $X_v$ by removing its intersection with all boundary components which do not contain $X_v$. It sits in a fibre diagram (see Section~\ref{sec: strata in Artin fans}):
\bcd
X_v^\circ \ar[r,hook] \ar[d] \ar[rd,phantom,"\square"] & X \ar[d] \\
\Bcal T_{\sigma_v} \ar[r,hook] & \Acal_\Sigma.
\ecd
The morphism $X_v^\circ \to \Bcal T_{\sigma_v}$ gives a mixing collection on $X_v^\circ$ with lattice $N_{\sigma_v}$. We refer to this as the \textbf{standard mixing collection}.

When $\sigma_v$ is smooth, the lattice $N_{\sigma_v}$ has a natural basis given by the primitive ray generators, and the mixing collection is equivalent to the list $\OO_X(D_\rho)|_{X_v^\circ}$ for $\rho \in \sigma_v(1)$. In this case there is a natural splitting of the normal bundle
\[ N_{X_v^\circ|X} = \bigoplus_{\rho \in \sigma_v(1)} \OO_X(D_\rho)|_{X_v^\circ}.\]

\subsection{Torus bundles $Y_v^\circ$} \label{sec: Yvcirc}
We first consider the locally-closed stratum
\[ Y_v^\circ \hookrightarrow Y_v.\] 
This is the open subvariety of $Y_v$ obtained by removing its intersection with all boundary components of $\Xcal_\Upsilon$ which do not contain $Y_v$.
\begin{proposition} \label{lem: Yvcirc as torus bundle} The restricted morphism
\[ \rho_v \colon Y_v^\circ \to X_v^\circ \]
is a principal torus bundle, with structure group $T_{\sigma_v}$. It is the principal bundle corresponding to the standard mixing collection of Section~\ref{sec: standard mixing collection}.
\end{proposition}

\begin{remark} The identification of the structure group with $T_{\sigma_v}$ is part of \cite[Theorem~1.8]{CN21}.\end{remark}

\begin{proof} The cone $\sigma_v \times \tau \in \Sigma \times \tau$ defines (see Section~\ref{sec: strata in Artin fans}) an open substack $\Acal_{\sigma_v \times \tau} \hookrightarrow \Acal_{\Sigma \times \tau}$. The composition
\[ \Bcal T_{\sigma_v \times \tau} \hookrightarrow \Acal_{\sigma_v \times \tau} \hookrightarrow \Acal_{\Sigma \times \tau} \]
defines a locally-closed substack whose pullback is the corresponding locally-closed stratum
\bcd
X_v^\circ \ar[r,hook] \ar[d] \ar[rd,phantom,"\square"] & X \times U_\tau \ar[d] \\
\Bcal T_{\sigma_v \times \tau} \ar[r,hook] & \Acal_{\Sigma \times \tau}.
\ecd
Similarly the cone $\omega_v \in \Upsilon$ defines a locally-closed substack $\Bcal T_{\omega_v} \hookrightarrow \Acal_\Upsilon$ whose pullback is the corresponding locally-closed stratum $Y_v^\circ \hookrightarrow \Xcal_\Upsilon$. We obtain a diagram
\bcd
Y_v^\circ \ar[rr,hook] \ar[dd] \ar[rd] & & \Xcal_\Upsilon \ar[dd] \ar[rd] \\
& X_v^\circ \ar[rr,hook,crossing over]  & & X \times U_\tau \ar[dd] \\
\Bcal T_{\omega_v} \ar[rr,hook] \ar[rd] & & \Acal_\Upsilon \ar[rd] \\
& \Bcal T_{\sigma_v \times \tau} \ar[from=uu, crossing over]\ar[rr,hook] & & \Acal_{\Sigma \times \tau}.
\ecd
The back and right faces are cartesian, hence so is their composition. This coincides with the composition of the left and front faces. Since the front face is also cartesian, it follows that the left face is cartesian:
\bcd
Y_v^\circ \ar[d] \ar[r] \ar[rd,phantom,"\square"] & \Bcal T_{\omega_v} \ar[d] \\
X_v^\circ \ar[r] & \Bcal T_{\sigma_v \times \tau}.
\ecd
By Proposition~\ref{tropical position map gives identification of lattices} we have a natural isomorphism $T_{\sigma_v \times \tau} / T_{\omega_v} = T_{\sigma_v}$. Applying Lemma~\ref{lem: BH to BG a principal bundle}, we obtain a cartesian square
\bcd
Y_v^\circ \ar[r] \ar[d] \ar[rd,phantom,"\square"] & \Bcal 0 \ar[d] \\
X_v^\circ \ar[r] & \Bcal T_{\sigma_v}
\ecd
where $\Bcal 0 =\Speck \to \Bcal T_{\sigma_v}$ is the universal principal bundle and $X_v^\circ \to \Bcal T_{\sigma_v}$ is the standard mixing collection of Section~\ref{sec: standard mixing collection}.
\end{proof}

\subsection{Toric variety bundles $Y_v^\bullet$} \label{sec: Yvbullet}
Now consider the open subset of $Y_v$ given by
\[ Y_v^\bullet \colonequals \rho_v^{-1}(X_v^\circ).\]
Note that if $\sigma_v \in \Sigma$ is maximal then $X_v^\circ=X_v$ and so $Y_v^\bullet=Y_v$. In general there is a sequence of open inclusions
\[ Y_v^\circ \hookrightarrow Y_v^\bullet \hookrightarrow Y_v.\] 

\begin{remark} The inclusion $Y_v^\circ \hookrightarrow Y_v^\bullet$ is a fibrewise toric compactification of the principal torus bundle $Y_v^\circ \to X_v^\circ$. If the open subdivision $\Upsilon \to \Sigma \times \tau$ is bijective on supports then the morphism $Y_v^\bullet \to X_v^\circ$ is proper. In contrast, $Y_v^\circ \to X_v^\circ$ is proper only if it is an isomorphism.\end{remark}

 In this section we prove that $Y_v^\bullet$ is obtained from $X_v^\circ$ by applying the mixing construction (Construction~\ref{mixing construction}). We first define the appropriate fibre fan. The subscheme $Y_v^\bullet \hookrightarrow \Xcal_\Upsilon$ is the union of the locally-closed strata indexed by cones in the following set
\[\ \Psi_v \colonequals \left \{ \omega \in \Upsilon \mid \omega_v \subseteq \omega \text{ and } \rtrop(\omega) \subseteq \sigma_v \right \}. \]
Indeed, for $\omega \in \Psi_v$ we have $\tau=\ptrop(\omega_v) \subseteq \ptrop(\omega)$ and so $\ptrop(\omega)=\tau$ which guarantees that $\omega$ indexes a stratum of the central fibre. Similarly, the conditions $\omega_v \subseteq \omega$ and $\rtrop(\omega)\subseteq\sigma_v$ together imply that $\sigma_v \in \Sigma$ is the minimal cone containing $\rtrop(\omega)$. This guarantees that the locally-closed stratum corresponding to $\omega$ maps to $X_v^\circ$ under $\rho_v$.


For every $\omega \in \Psi_v$ the restriction of the open subdivision $\rtrop \times \ptrop \colon \Upsilon \to \Sigma \times \tau$ produces an inclusion
\[ \omega \subseteq (N_{\sigma_v} \times N_\tau) \otimes \R.\] 
By Proposition~\ref{tropical position map gives identification of lattices} we have a natural isomorphism $(N_{\sigma_v} \times N_\tau)/N_{\omega_v} \cong N_{\sigma_v}$. Let
\[ \omega/\omega_v \subseteq N_{\sigma_v} \otimes \R\]
denote the image of $\omega$ in the quotient. This is strictly convex because $\omega_v \subseteq \omega$ is a face.

\begin{definition}\label{def: fibre fan}
The \textbf{fibre fan} associated to $Y_v^\bullet \to X_v^\circ$ is the fan 
\[ \Phi_v \colonequals \{ \omega/\omega_v \mid \omega \in \Psi_v \}\]
in the lattice $N_{\sigma_v}$.
\end{definition}

\begin{remark} The fibre fan is visible in the polyhedral complex $\Upsilon_{\f}$. Restrict this complex to a small neighbourhood around the vertex $v$. Declare $v$ to be the origin, and extend all local polyhedra radially out from this point. Intersecting with $N_{\sigma_v}$ then gives the fan $\Phi_v$. \end{remark}

We now establish the main result constructing $Y_v^\bullet$ from $X_v^\circ$.
\begin{theorem}[Theorem~\ref{thm: Yvbullet introduction}] \label{thm: Yvbullet description} The morphism
\[ \rho_v \colon Y_v^\bullet \to X_v^\circ \]
is a toric variety bundle. It coincides with the output of the mixing construction (Construction~\ref{mixing construction}), with fibre fan $\Phi_v$ (Definition~\ref{def: fibre fan}) and the standard mixing collection (Section~\ref{sec: standard mixing collection}).\end{theorem}

\begin{proof}
Recall the discussion in Sections~\ref{sec: isotropic cone complex}~and~\ref{sec: strata in Artin fans}. To each $\omega \in \Psi_v$ we associate the isotropic cone
\[ (\omega/\omega_v,N_{\omega/\omega_v},N_{\omega}) \]
and hence view $\Psi_v$ as an isotropic cone complex. It is a subcomplex of $\Star(\omega_v,\Upsilon)$. The sequence of embeddings $Y_v^\bullet \hookrightarrow Y_v \hookrightarrow \Xcal_{\Upsilon}$ corresponds to the sequence of isotropic Artin fans $\Bcal_{\Psi_v} \hookrightarrow \Bcal_{\Star(\omega_v,\Upsilon)} \hookrightarrow \Acal_{\Upsilon}$. By the definition of $Y_v^\bullet$ there is a cartesian square
\bcd
Y_v^\bullet \ar[r,hook] \ar[d] \ar[rd,phantom,"\square"] & \Xcal_\Upsilon \ar[d] \\
\Bcal_{\Psi_v} \ar[r,hook] & \Acal_\Upsilon.
\ecd
For $\omega \in \Psi_v$ we have $\omega \subseteq \sigma_v \times \tau$ and so we obtain a homomorphism $T_\omega \to T_{\sigma_v \times \tau}$. Recall that $U_{\omega/\omega_v}$ denotes the affine toric variety associated to the cone $\omega/\omega_v$ in $N_{\omega/\omega_v}$. The compositions $[U_{\omega/\omega_v}/T_\omega] \to \Bcal T_\omega \to \Bcal T_{\sigma_v \times \tau}$ glue along face inclusions to produce a global morphism
\[ \Bcal_{\Psi_v} \to \Bcal T_{\sigma_v \times \tau} \]
which corresponds to the map $\Psi_v \to N_{\sigma_v \times \tau}$ of isotropic cone complexes. We obtain a cube
\begin{equation} \label{eqn: cartesian cube proof of Yvbullet}
\begin{tikzcd}
Y_v^\bullet \ar[rr,hook] \ar[dd] \ar[rd] & & \Xcal_\Upsilon \ar[dd] \ar[rd] \\
& X_v^\circ \ar[rr,hook,crossing over]  & & X \times U_\tau \ar[dd] \\
\Bcal_{\Psi_v} \ar[rr,hook] \ar[rd] & & \Acal_\Upsilon \ar[rd] \\
& \Bcal T_{\sigma_v \times \tau} \ar[from=uu, crossing over]\ar[rr,hook] & & \Acal_{\Sigma \times \tau}
\end{tikzcd}
\end{equation}
and a diagram chase as in the proof of Proposition~\ref{lem: Yvcirc as torus bundle} shows that the left face is cartesian. We now show that there is a fibre square
\begin{equation} \label{eqn: cartesian square isotropic Artin fans}
\begin{tikzcd}
\Bcal_{\Psi_v} \ar[d] \ar[r] \ar[rd,phantom,"\square"] & \Acal_{\Phi_v} \ar[d] \\
\Bcal T_{\sigma_v \times \tau} \ar[r] & \Bcal T_{\sigma_v}
\end{tikzcd}
\end{equation}
where the map $\Bcal T_{\sigma_v \times \tau} \to \Bcal T_{\sigma_v}$ is the gerbe banded by $T_{\omega_v}$ induced by the isomorphism $(N_{\sigma_v} \times N_\tau)/N_{\omega_v} = N_{\sigma_v}$ (see Proposition~\ref{tropical position map gives identification of lattices} and Lemma~\ref{lem: BH to BG a principal bundle}).

For every $\omega \in \Psi_v$ there is a homomorphism of tori $T_{\omega} \to T_{\omega/\omega_v}$ through which the action $T_{\omega} \acts U_{\omega/\omega_v}$ is defined. This gives a morphism
\[ [U_{\omega/\omega_v}/T_\omega] \to [U_{\omega/\omega_v}/T_{\omega/\omega_v}] \]
compatible along face inclusions. Globalising, we obtain $\Bcal_{\Psi_v} \to \Acal_{\Phi_v}$ which we combine with the map $\Bcal_{\Psi_v} \to \Bcal T_{\sigma_v \times \tau}$ constructed above to obtain a morphism
\begin{equation} \label{eqn: map from Bpsiv to fibre product} \Bcal_{\Psi_v} \to \Bcal T_{\sigma_v \times \tau} \times_{\Bcal T_{\sigma_v}} \Acal_{\Phi_v}. \end{equation}
It can be checked locally on the target that the morphism $\Bcal_{\Psi_v} \to \Acal_{\Phi_v}$ is a gerbe banded by $T_{\omega_v}$. On the other hand the morphism $\Bcal T_{\sigma_v \times \tau} \times_{\Bcal T_{\sigma_v}} \Acal_{\Phi_v} \to \Acal_{\Phi_v}$ is a gerbe banded by $T_{\omega_v}$ by Lemma~\ref{lem: BH to BG a principal bundle}. A morphism between two gerbes banded by the same group is automatically an isomorphism (see e.g. \cite[proof of Proposition~5.11]{Brochard}). We conclude that \eqref{eqn: cartesian square isotropic Artin fans} is cartesian (we thank the anonymous referee for suggesting this argument). Composing with the left face of \eqref{eqn: cartesian cube proof of Yvbullet} we obtain
\begin{equation} \label{eqn: Yvbullet over Xvcirc as pullback of universal TVB}
\begin{tikzcd}
Y_v^\bullet \ar[d] \ar[r] \ar[rd,phantom,"\square"] & \Acal_{\Phi_v} \ar[d] \\
X_v^\circ \ar[r] & \Bcal T_{\sigma_v}.
\end{tikzcd}
\end{equation}
It follows from Lemma~\ref{lem: universal toric variety bundle} that $Y_v^\bullet \to X_v^\circ$ is a toric variety bundle, with fibre fan $\Phi_v$ and mixing collection $X_v^\circ \to \Bcal T_{\sigma_v}$.
\end{proof}

\begin{remark}\label{rmk: cut and paste}
Theorem~\ref{thm: Yvbullet description} generalises: for every locally-closed stratum $S \hookrightarrow X_v$ the restriction
\[ \rho_v \colon \rho_v^{-1}(S) \to S \]
is a finite union of toric variety bundles, glued along toric strata. This provides a cut-and-paste description of the morphism $Y_v \to X_v$, as a stratified union of toric variety bundles. See \cite[Proposition~2.1.4]{HuLiuYau} for a parallel in the toric setting. 
\end{remark}

\subsection{Components $Y_v$} \label{sec: Yv closed stratum} We now consider the entire component $Y_v$. We emphasise that if $\sigma_v$ is maximal, then $X_v=X_v^\circ$ and $Y_v=Y_v^\bullet$. In this case Theorem~\ref{thm: Yvbullet description} gives a complete description of $Y_v$. The following discussion is significantly more delicate, and is only required when $\sigma_v$ is not maximal.

We thank the anonymous referee for supplying a crucial suggestion allowing us to remove a superfluous smoothness hypothesis in this section, and identifying an important conceptual mistake in an earlier version of the argument.

\subsubsection{Bundle criterion} As discussed in Section~\ref{sec: examples}, the morphism $Y_v \to X_v$ is not always a toric variety bundle. We give a sufficient criterion, generalising the corresponding criterion in toric geometry \cite[Theorem~3.3.19]{CLS}.

Recall the fibre fan $\Phi_v$ of Definition~\ref{def: fibre fan}, and note that there is a natural subcomplex inclusion:
\[ \Phi_v \hookrightarrow \Upsilon/\omega_v. \]

\begin{theorem}[Theorem~\ref{thm: Yv introduction}] \label{thm: Yv description}
The morphism $Y_v \to X_v$ is a toric variety bundle if there exists an embedding of cone complexes (i.e. a subcomplex inclusion)
\[ s \colon \Sigma/\sigma_v \hookrightarrow \Upsilon/\omega_v \]
which is a section of the projection $\Upsilon/\omega_v \to \Sigma/\sigma_v$ and is such that:
\begin{equation} \label{eqn: isomorphism bundle criterion} \Upsilon/\omega_v = s(\Sigma/\sigma_v) + \Phi_v. \end{equation}
\end{theorem}

Condition \eqref{eqn: isomorphism bundle criterion} means that every $\omega/\omega_v \in \Upsilon/\omega_v$ has faces belonging to the given subcomplexes
\[ s(\sigma/\sigma_v) \in s(\Sigma/\sigma_v), \qquad \theta/\omega_v \in \Phi_v, \]
such that $\omega/\omega_v = s(\sigma/\sigma_v) + \theta/\omega_v$, and conversely every such pair of faces produces a cone of $\Upsilon/\omega_v$.
\begin{remark} The condition can be rephrased as stating that there exists an isomorphism
\[ \alpha \colon \Upsilon/\omega_v \cong \Sigma/\sigma_v \times \Phi_v \]
such that the induced projection
\[ \Upsilon/\omega_v \xrightarrow{\alpha} \Sigma/\sigma_v \times \Phi_v \to \Sigma/\sigma_v \]
agrees with the natural projection \eqref{eqn: commuting square boundary strata}, and the induced inclusion
\[ \Phi_v \hookrightarrow \Sigma/\sigma_v \times \Phi_v \xrightarrow{\alpha^{-1}} \Upsilon/\omega_v \]
agrees  with the natural inclusion (Definition~\ref{def: fibre fan}).
\end{remark}

\begin{proof}[Proof of Theorem~\ref{thm: Yv description}] There is a diagram
\bcd
Y_v \ar[rr,hook] \ar[dd] \ar[rd] & & \Xcal_\Upsilon \ar[dd] \ar[rd] \\
& X_v \ar[rr,hook,crossing over]  & & X \times U_\tau \ar[dd] \\
\Bcal_{\Star(\omega_v,\Upsilon)} \ar[rr,hook] \ar[rd] & & \Acal_\Upsilon \ar[rd] \\
& \Bcal_{\Star(\sigma_v \times \tau, \Sigma \times \tau)} \ar[from=uu, crossing over]\ar[rr,hook] & & \Acal_{\Sigma \times \tau}
\ecd
and as in the proof of Proposition~\ref{lem: Yvcirc as torus bundle} we see that the left face is cartesian. It is thus sufficient to prove that the morphism
\begin{equation} \label{eqn: map rhov isotropic Artin fans} \rho_v \colon \Bcal_{\Star(\omega_v,\Upsilon)} \to  \Bcal_{\Star(\sigma_v \times \tau, \Sigma \times \tau)} \end{equation}
is a toric variety bundle. On the domain, there is a $T_{\omega_v}$-gerbe (see Lemma~\ref{lem: isotropic Artin fan as a gerbe})
\begin{equation} \label{eqn: Yv proof first map} \Bcal_{\Star(\omega_v,\Upsilon)} \to \Acal_{\Upsilon/\omega_v} = \Acal_{\Sigma/\sigma_v} \times \Acal_{\Phi_v} \end{equation}
where the identity follows from \eqref{eqn: isomorphism bundle criterion}.

On the codomain, first note that $\Star(\sigma_v\! \times\! \tau, \Sigma\! \times\! \tau)$ is the following isotropic cone complex:
\begin{equation} \label{eqn: Star as an isotropic cone complex} \Star(\sigma_v\! \times\! \tau, \Sigma\! \times\! \tau) = \{ (\sigma/\sigma_v, N_{\sigma/\sigma_v}, N_{\sigma \times \tau}) : \sigma \in \Sigma, \sigma_v \subseteq \sigma \} \end{equation}
where the surjection of lattices is given by the composite $N_{\sigma \times \tau} \to N_\sigma \to N_{\sigma/\sigma_v}$. Given $\sigma \in \Sigma$ with $\sigma_v \subseteq \sigma$ consider the following morphism of short exact sequences
\bcd
0 \ar[r] & N_{\omega_v} \ar[r] \ar[d] & N_{\sigma \times \tau} \ar[r] \ar[d,equals] & N_{(\sigma \times \tau)/\omega_v} \ar[r] \ar[d] & 0 \\
0 \ar[r] & N_{\sigma_v \times \tau} \ar[r] & N_{\sigma \times \tau} \ar[r] & N_{\sigma/\sigma_v} \ar[r] & 0
\ecd
where we use the natural inclusion $\omega_v \subseteq \sigma_v \times \tau$ discussed in Section~\ref{sec: tropical position}. The snake lemma then shows that the right vertical map is surjective, with kernel equal to
\[ N_{\sigma_v \times \tau}/N_{\omega_v} \cong N_{\sigma_v} \] 
where the isomorphism is given by Proposition~\ref{tropical position map gives identification of lattices}. We thus obtain a short exact sequence:
\begin{equation} \label{eqn: proof Yv exact sequence} 0 \to N_{\sigma_v} \to N_{(\sigma \times \tau)/\omega_v} \to N_{\sigma/\sigma_v} \to 0.\end{equation}
Define the isotropic cone complex:
\begin{equation} \label{eqn: Omega as an isotropic cone complex} \Omega_v \colonequals \{ (\sigma/\sigma_v, N_{\sigma/\sigma_v}, N_{(\sigma \times \tau)/\omega_v}) : \sigma \in \Sigma, \sigma_v \subseteq \sigma \} \end{equation}
where the surjection of lattices is given by \eqref{eqn: proof Yv exact sequence}. The underlying cone complex is $\Sigma/\sigma_v$ and so there is a natural morphism
\begin{equation} \label{eqn: Yv proof second map} \Bcal_{\Omega_v} \to \Acal_{\Sigma/\sigma_v} \end{equation}
which is a $T_{\sigma_v}$-gerbe by \eqref{eqn: proof Yv exact sequence} and Lemma~\ref{lem: isotropic Artin fan as a gerbe}. Comparing \eqref{eqn: Star as an isotropic cone complex} and \eqref{eqn: Omega as an isotropic cone complex}, we see that there is a natural morphism of isotropic cone complexes $\Star(\sigma_v\! \times\! \tau, \Sigma\! \times\! \tau) \to \Omega_v$ giving rise to a morphism of isotropic Artin fans
\begin{equation} \label{eqn: Yv proof third map} \Bcal_{\Star(\sigma_v \times \tau, \Sigma \times \tau)}  \to \Bcal_{\Omega_v} \end{equation}
which is a $T_{\omega_v}$-gerbe by Lemma~\ref{lem: two quotients compared}, since locally on the base it takes the form:
\[ [U_{\sigma/\sigma_v}/T_{\sigma \times \tau}] \to [U_{\sigma/\sigma_v}/T_{(\sigma \times \tau)/\omega_v}]. \]
We now construct a morphism $\Upsilon/\omega_v \to \Omega_v$. As an isotropic cone complex:
\[ \Upsilon/\omega_v = \{ (\omega/\omega_v, N_{\omega/\omega_v}, N_{\omega/\omega_v}) : \omega \in \Upsilon, \omega_v \subseteq \omega \}. \]
Given $\omega \in \Upsilon$ with $\omega_v \subseteq \omega$ there is a unique minimal cone $\sigma \in \Sigma$ with $\sigma_v \subseteq \sigma$ such that $\omega \subseteq \sigma \times \tau$. This gives inclusions of lattices
\bcd
N_{\omega_v} \ar[r,hook] \ar[d,hook] & N_{\omega} \ar[d,hook] \\
N_{\sigma_v \times \tau} \ar[r,hook] & N_{\sigma \times \tau}
\ecd
which give rise to an inclusion
\[ N_{\omega/\omega_v} \hookrightarrow N_{(\sigma \times \tau)/\omega_v}.\]
Composing with the surjection in \eqref{eqn: proof Yv exact sequence} produces a map $N_{\omega/\omega_v} \to N_{\sigma/\sigma_v}$ which respects the cones, and we thus obtain a morphism:
\begin{equation} \label{eqn: Yv proof fourth map} \Upsilon/\omega_v \to \Omega_v. \end{equation}
We arrive at the following square, which commutes by a direct examination of lattices:
\bcd
\Bcal_{\Star(\omega_v,\Upsilon}) \ar[d,"\rho_v"] \ar[r,"{\eqref{eqn: Yv proof first map}}"] & \Acal_{\Upsilon/\omega_v} \ar[d,"{\eqref{eqn: Yv proof fourth map}}"] \\
\Bcal_{\Star(\sigma_v \times \tau, \Sigma \times \tau)} \ar[r,"{\eqref{eqn: Yv proof third map}}"] & \Bcal_{\Omega_v}
\ecd
Both horizontal morphisms are gerbes banded by $T_{\omega_v}$. Since a morphism between gerbes banded by the same group is automatically an isomorphism, it follows that this square is cartesian.

A direct examination of lattices (appealing to the equivalence of categories \cite[Theorem~3]{CavalieriChanUlirschWise} between cone stacks and Artin fans) shows that the composite of \eqref{eqn: Yv proof fourth map} with \eqref{eqn: Yv proof second map} coincides with the natural projection $\Acal_{\Upsilon/\omega_v} \to \Acal_{\Sigma/\sigma_v}$. Recall that we are given an embedding $\Sigma/\sigma_v \hookrightarrow \Upsilon/\omega_v$ which is a section of the projection. It follows that the composite
\[ \Acal_{\Sigma/\sigma_v} \hookrightarrow \Acal_{\Upsilon/\omega_v} \xrightarrow{\text{\eqref{eqn: Yv proof fourth map}}} \Bcal_{\Omega_v} \xrightarrow{\text{\eqref{eqn: Yv proof second map}}} \Acal_{\Sigma/\sigma_v} \]
is the identity. We thus obtain a preferred trivialisation of the $T_{\sigma_v}$-gerbe $\Bcal_{\Omega_v}$ over $\Acal_{\Sigma/\sigma_v}$. This produces the following diagram
\begin{equation} \label{eqn: Yv proof big diagram}
\begin{tikzcd}
\Bcal_{\Star(\omega_v,\Upsilon}) \ar[d,"\rho_v"] \ar[r] \ar[rd,phantom,"\square"] & \Acal_{\Upsilon/\omega_v} \ar[d] \ar[r,equals] & \Acal_{\Sigma/\sigma_v} \times \Acal_{\Phi_v} \ar[d] \\
\Bcal_{\Star(\sigma_v \times \tau, \Sigma \times \tau)} \ar[r] & \Bcal_{\Omega_v} \ar[r,"\cong"] & \Acal_{\Sigma/\sigma_v} \times \Bcal T_{\sigma_v}
\end{tikzcd}
\end{equation}
where the right vertical map is the external product of the identity $\Acal_{\Sigma/\sigma_v} \to \Acal_{\Sigma/\sigma_v}$ and the natural map $\Acal_{\Phi_v} \to \Bcal T_{\sigma_v}$.

We now show that the right square commutes, and for this it suffices to work locally. Each $\omega/\omega_v \in \Upsilon/\omega_v$ contains faces $\sigma/\sigma_v \in \Sigma/\sigma_v$ and $\theta/\omega_v \in \Phi_v$ such that:
\[ \omega/\omega_v = \sigma/\sigma_v \times \theta/\omega_v.\]
Restricting to the corresponding open set in $\Acal_{\Sigma/\sigma_v} \times \Acal_{\Phi_v}$ we must show that the associated diagram of lattices commutes:
\begin{equation}
\label{eqn: commuting square lattices}
\begin{tikzcd}
N_{\omega/\omega_v} \ar[d] & N_{\sigma/\sigma_v} \times N_{\theta/\omega_v} \ar[d] \ar[l,"=" above] \\
N_{(\sigma \times \tau)/\omega_v} \ar[r,"\cong"] & N_{\sigma/\sigma_v} \times N_{\sigma_v}.
\end{tikzcd}
\end{equation}
Since $\Sigma/\sigma_v \to \Upsilon/\omega_v \to \Sigma/\sigma_v$ is the identity, \eqref{eqn: commuting square lattices} commutes when we restrict to $N_{\sigma/\sigma_v}$ in the domain and project onto $N_{\sigma/\sigma_v}$ in the codomain. On the other hand, recall that the isomorphism
\[ N_{(\sigma \times \tau)/\omega_v} \cong N_{\sigma/\sigma_v} \times N_{\sigma_v} \]
is induced by the section $N_{\sigma/\sigma_v} \to N_{(\sigma \times \tau)/\omega_v}$ which splits the exact sequence \eqref{eqn: proof Yv exact sequence}. It follows that \eqref{eqn: commuting square lattices} commutes when we restrict to $N_{\sigma/\sigma_v}$ in the domain and project onto $N_{\sigma_v}$ in the codomain, since both routes produce the zero map.

On the other hand commutativity when restricted to $N_{\theta/\omega_v}$ follows from the commutativity of:
\bcd
N_{\omega/\omega_v} \ar[d] & & N_{\theta/\omega_v} \ar[ll] \ar[d] \\
N_{(\sigma \times \tau)/\omega_v} & N_{(\sigma_v \times \tau)/\omega_v} \ar[l,hook] & N_{\sigma_v} \ar[l,"\cong" above]
\ecd
noting the natural isomorphism $ N_{\sigma_v} \cong N_{(\sigma_v \times \tau)/\omega_v}$ (Proposition~\ref{tropical position map gives identification of lattices}) which was used to construct the map $N_{\theta/\omega_v} \to N_{\sigma_v}$ in the first place (Definition~\ref{def: fibre fan}). We conclude that \eqref{eqn: commuting square lattices}, and therefore \eqref{eqn: Yv proof big diagram}, commutes. From this we obtain the cartesian square
\bcd
\Bcal_{\Star(\omega_v,\Upsilon)} \ar[r] \ar[d,"\rho_v"] \ar[rd,phantom,"\square"] & \Acal_{\Phi_v} \ar[d] \\
\Bcal_{\Star(\sigma_v \times \tau,\Sigma \times \tau)} \ar[r] & \Bcal T_{\sigma_v}
\ecd
which shows that $\rho_v$ is a toric variety bundle, as required.
\end{proof}

\begin{remark} The section $s$ in Theorem~\ref{thm: Yv description} is unique if it exists. Indeed suppose $s$ exists and fix a cone $\sigma/\sigma_v \in \Sigma/\sigma_v$. Consider those cones of $\Upsilon/\omega_v$ which map to $\sigma/\sigma_v$ or one of its faces. These form a subcomplex of $\Upsilon/\omega_v$, and by the product structure this subcomplex is isomorphic to $\sigma/\sigma_v \times \Phi_v$. In particular, only a single cone of this subcomplex is mapped isomorphically onto $\sigma/\sigma_v$. There is therefore a unique choice for the section $s$ restricted to $\sigma/\sigma_v$.
\end{remark}

\begin{remark} We do not know whether the converse to Theorem~\ref{thm: Yv description} holds. We now explain the difficulty. Following Construction~\ref{mixing construction}, a \textbf{toric variety bundle structure} for a morphism $Y \to X$ consists of an open subset $P \subseteq Y$ with $P \to X$ a principal $T_Z$-bundle, and an isomorphism $Y \cong (P \times Z)/T_Z$. Following Section~\ref{sec: universal toric variety bundle} this is equivalent to the data of a $2$-cartesian square:
\[
\begin{tikzcd}
Y \ar[d] \ar[r] \ar[rd,phantom,"\square"] & {[Z/T_Z]} \ar[d] \\
X \ar[r] & \Bcal T_Z.
\end{tikzcd}
\]
Pullbacks of toric variety bundle structures are defined by composing such squares. A toric variety bundle typically admits infinitely many toric variety bundle structures.
 
Returning to Theorem~\ref{thm: Yv description}, there exist toric variety bundle structures for $Y_v \to X_v$ which are not pulled back from any toric variety bundle structure for the corresponding morphism
\[ \Bcal_{\Star(\omega_v,\Upsilon)} \to  \Bcal_{\Star(\sigma_v \times \tau, \Sigma \times \tau)} \]
of isotropic Artin fans \eqref{eqn: map rhov isotropic Artin fans}. This is the case, for instance, if pulling back the open subset $P \subseteq Y_v$ along $Y_v^\bullet \hookrightarrow Y_v$ produces a different open subset than $Y_v^\circ \subseteq Y_v^\bullet$. This forecloses any conclusions involving the associated cone complexes.

To establish the converse to Theorem~\ref{thm: Yv description}, we would therefore need to argue that if $Y_v \to X_v$ is a toric variety bundle, then it in fact admits a toric variety bundle structure which is pulled back from the isotropic Artin fans. This may be true, but proving it seems to require new insights.
\end{remark}

\begin{remark}
If $Y_v \to X_v$ is not a toric variety bundle, it can be transformed into one by toroidal modifications of the source and target. This is the strongest statement one can hope for in general: for instance if $X_v=X$ then the morphism $Y_v \to X_v$ can be an arbitrary toroidal modification.
\end{remark}

\subsubsection{Determining the mixing collection}\label{sec: determining the mixing collection for Yv}

Assuming that the criterion of Theorem~\ref{thm: Yv description} is satisfied, the morphism $Y_v \to X_v$ is a toric variety bundle and is therefore obtained by applying the mixing construction (Construction~\ref{mixing construction}). The fibre fan is simply $\Phi_v$ as given in Definition~\ref{def: fibre fan}. The mixing collection is given by the composite:
\[ X_v \to \Bcal_{\Star(\sigma_v \times \tau,\Sigma \times \tau)} \to \Bcal T_{\sigma_v} \]
where the second map was constructed in the proof of Theorem~\ref{thm: Yv description}, using the subcomplex inclusion $\Sigma/\sigma_v \hookrightarrow \Upsilon/\omega_v$ to trivialise the $T_{\sigma_v}$-gerbe $\Bcal_{\Omega_v}$ over $\Acal_{\Sigma/\sigma_v}$.

When $\Sigma$ is a smooth cone complex this mixing collection can be calculated explicitly. In this case the lattice $N_{\sigma_v}$ has a natural basis given by the primitive generators of the rays $\rho \in \sigma_v(1)$. The list of line bundles $\OO_{X_v}(D_\rho)$ for $\rho \in \sigma_v(1)$ defines a mixing collection on $X_v$. This is precisely the mixing collection corresponding to the composite
\[ X_v \to \Bcal_{\Star(\sigma_v,\Sigma)} = \Acal_{\Sigma/\sigma_v} \times \Bcal T_{\sigma_v} \to \Bcal T_{\sigma_v}\]
where smoothness of $\Sigma$ ensures a canonical isomorphism $\Bcal_{\Star(\sigma_v,\Sigma)} = \Acal_{\Sigma/\sigma_v} \times \Bcal T_{\sigma_v}$ (see Remark~\ref{rmk: isotropic Artin fan smooth complexes}). We refer to this as the \textbf{standard mixing collection}. The mixing collection inducing the toric variety bundle $Y_v \to X_v$ differs from the standard mixing collection in general. The difference is a sum of divisors supported on the boundary $X_v\setminus X_v^\circ$.

This difference can be calculated explicitly using Lemma~\ref{lem: identifying mixing collection}. For $\rho \in \sigma_v(1)$ we must calculate the pullback of the line bundle $\OO_{X_v}(D_\rho)$ along the projection $Y_v \to X_v$. This can be expressed in terms of piecewise-linear functions on $\Sigma$ and $\Upsilon$, as we now explain.

The divisor $D_\rho \subseteq X$ corresponds to a piecewise-linear function on $\Sigma$ which pulls back to a piecewise-linear function on $\Upsilon$. Translating by the pullback of a linear function on $\tau$ we obtain a piecewise-linear function on $\Upsilon$ which is identically zero on $\omega_v$; this is always possible because $\omega_v \to \tau$ is an isomorphism. 

The corresponding toroidal divisor in $\Xcal_\Upsilon$ intersects the stratum $Y_v$ transversely, so we can directly describe its restriction to $Y_v$. The result is a sum of horizontal toric divisors on $Y_v$, together with divisors pulled back from $X_v$. The former give the left-hand side of \eqref{eqn: divisor relation in toric bundle}, the latter the right-hand side. Thus, the mixing collection is determined.

\begin{example} Take $X$ a smooth variety and $D=D_1+D_2$ the union of two smooth divisors with connected nonempty intersection. Consider the following tropical expansion over $\tau=\RR_{\geq 0}$:
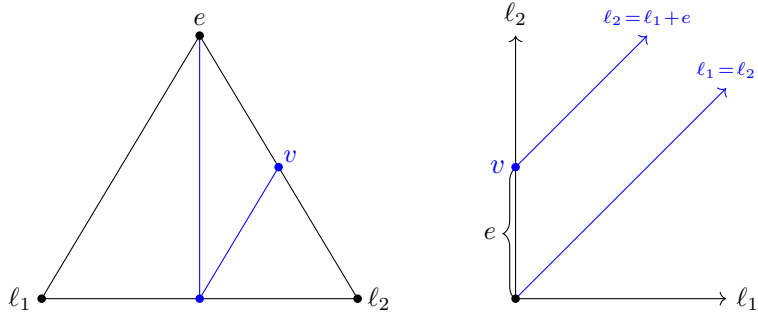
\[
\begin{tikzpicture}[scale=0.7]

\draw[fill=black] (0,0) circle[radius=2pt];
\draw (0,0) node[left]{\small$\ell_1$};
\draw (0,0) -- (6,0);
\draw[fill=black] (6,0) circle[radius=2pt];
\draw (6,0) node[right]{\small$\ell_2$};
\draw (6,0) -- (3,5);
\draw (3,5) node[above]{\small$e$};
\draw (3,5) -- (0,0);
\draw[fill=blue,blue] (3,0) circle[radius=2pt];
\draw[blue] (3,0) -- (3,5);
\draw[blue] (3,0) -- (4.5,2.5);
\draw[fill=blue,blue] (4.5,2.5) circle[radius=2pt];
\draw[blue] (4.4,2.7) node[right]{\small$v$};
\draw[fill=black] (3,5) circle[radius=2pt];

\draw[->] (9,0) -- (13,0);
\draw (13,0) node[right]{\small$\ell_1$};
\draw[->] (9,0) -- (9,5);
\draw (9,5) node[above]{\small$\ell_2$};
\draw [decorate,decoration={brace,amplitude=5pt},xshift=0.4pt,yshift=-0.4pt] (9,0) --(9,2.5) node[black,midway,xshift=-0.35cm,yshift=-0 cm] {\small{$e$}};
\draw[fill=blue,blue] (9,2.5) circle[radius=2pt];
\draw[blue] (9,2.5) node[left]{$v$};
\draw[blue,->] (9,2.5) -- (11.5,5);
\draw[blue] (11.5,5) node[above]{\tiny$\ell_2\!=\!\ell_1\!+\!e$};
\draw[blue,->] (9,0) -- (13,4);
\draw[blue] (13,4) node[above]{\tiny$\ell_1\!=\!\ell_2$};
\draw[fill=black] (9,0) circle[radius=2pt];

\end{tikzpicture}
\]
As usual we have a degenerating family $\pi \colon \Xcal_\Upsilon \to U_\tau = \Aaff^{\! 1}$ and a fibrewise collapsing morphism $\rho \colon \Xcal_\Upsilon \to X$. By Theorem~\ref{thm: Yv description}, the component $Y_v$ of the central fibre is a $\PP^1$-bundle over $X_v=D_2$ and we let $\rho_v \colon Y_v \to X_v$ denote the bundle projection. To calculate the mixing collection, first name the toroidal divisors on $Y_v$ as follows
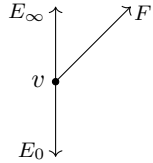
\[
\begin{tikzpicture}[scale=0.4];

\draw[fill=black] (9,2.5) circle[radius=3pt];
\draw (9,2.5) node[left]{\small$v$};
\draw[->] (9,2.5) -- (9,5);
\draw (9,4.8) node[left]{\scriptsize$E_\infty$};
\draw[->] (9,2.5) -- (9,0);
\draw (9,0.2) node[left]{\scriptsize$E_0$};
\draw[->] (9,2.5) -- (11.5,5);
\draw (11.9,4.8) node{\scriptsize$F$};

\end{tikzpicture}
\]
where $E_0$ and $E_\infty$ are the horizontal toric divisors of the $\PP^1$-bundle, and $F$ is the fibre over ${D_1\cap D_2}$. Note that $\OO_{Y_v}(F) = \rho_v^\star \OO_{X_v}(D_1)$. Consider on $\Upsilon$ the following piecewise-linear functions, represented visually by their slopes along rays:

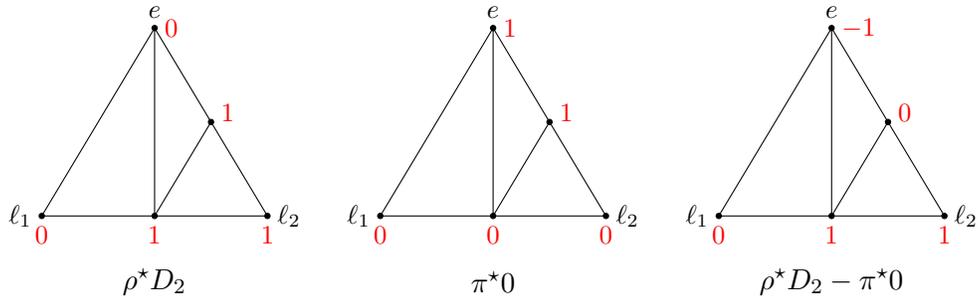
\[
\begin{tikzpicture}[scale=0.5]

\draw[fill=black] (3,0) circle[radius=2pt];
\draw (3,0) -- (3,5);
\draw (3,0) -- (4.5,2.5);
\draw[fill=black] (4.5,2.5) circle[radius=2pt];
\draw[fill=black] (0,0) circle[radius=2pt];
\draw (0,0) -- (6,0);
\draw[fill=black] (6,0) circle[radius=2pt];
\draw (6,0) -- (3,5);
\draw[fill=black] (3,5) circle[radius=2pt];
\draw (3,5) -- (0,0);

\draw (0,0) node[left]{\small$\ell_1$};
\draw (6,0) node[right]{\small$\ell_2$};
\draw (3,5) node[above]{\small$e$};

\draw[red] (0,0) node[below]{\small$0$};
\draw[red] (3,0) node[below]{\small$1$};
\draw[red] (6,0) node[below]{\small$1$};
\draw[red] (4.5,2.75) node[right]{\small$1$};
\draw[red] (3,5) node[right]{\small$0$};

\draw (3,-1.75) node{$\rho^\star D_2$};

\draw[fill=black] (12,0) circle[radius=2pt];
\draw (12,0) -- (12,5);
\draw (12,0) -- (13.5,2.5);
\draw[fill=black] (13.5,2.5) circle[radius=2pt];
\draw[fill=black] (9,0) circle[radius=2pt];
\draw (9,0) -- (15,0);
\draw[fill=black] (15,0) circle[radius=2pt];
\draw (15,0) -- (12,5);
\draw[fill=black] (12,5) circle[radius=2pt];
\draw (12,5) -- (9,0);

\draw (9,0) node[left]{\small$\ell_1$};
\draw (15,0) node[right]{\small$\ell_2$};
\draw (12,5) node[above]{\small$e$};

\draw[red] (9,0) node[below]{\small$0$};
\draw[red] (12,0) node[below]{\small$0$};
\draw[red] (15,0) node[below]{\small$0$};
\draw[red] (13.5,2.75) node[right]{\small$1$};
\draw[red] (12,5) node[right]{\small$1$};

\draw (12,-1.75) node{$\pi^\star 0$};

\draw[fill=black] (21,0) circle[radius=2pt];
\draw (21,0) -- (21,5);
\draw (21,0) -- (22.5,2.5);
\draw[fill=black] (22.5,2.5) circle[radius=2pt];
\draw[fill=black] (18,0) circle[radius=2pt];
\draw (18,0) -- (24,0);
\draw[fill=black] (24,0) circle[radius=2pt];
\draw (24,0) -- (21,5);
\draw[fill=black] (21,5) circle[radius=2pt];
\draw (21,5) -- (18,0);

\draw (18,0) node[left]{\small$\ell_1$};
\draw (24,0) node[right]{\small$\ell_2$};
\draw (21,5) node[above]{\small$e$};

\draw[red] (18,0) node[below]{\small$0$};
\draw[red] (21,0) node[below]{\small$1$};
\draw[red] (24,0) node[below]{\small$1$};
\draw[red] (22.5,2.75) node[right]{\small$0$};
\draw[red] (21,5) node[right]{\small$-1$};

\draw (21,-1.75) node{$\rho^\star D_2 - \pi^\star 0$};

\end{tikzpicture}
\]
Restricting the divisor $\rho^\star D_2 - \pi^\star 0$ from the total space $\Xcal_\Upsilon$ to the central fibre component $Y_v$ gives
\[ \rho_v^\star \OO_{X_v}(D_2) = \OO_{Y_v}(E_\infty-E_0+F) \]
which we rearrange to obtain
\[ \OO_{Y_v}(E_\infty-E_0) = \rho_v^\star \OO_{X_v}(D_2-D_1).\]
Combining Lemma~\ref{lem: identifying mixing collection} and Theorem~\ref{thm: TVB implies GIT} we conclude
\[ Y_v = \PP_{X_v}(\OO_{X_v}(D_2-D_1) \oplus \OO_{X_v}).\]
\end{example}

\footnotesize
\noindent \textbf{Data Availability Statement.} Data sharing is not applicable to this article as no datasets were generated or analysed during the current study.

\noindent \textbf{Conflict of Interest Statement.} On behalf of all authors, the corresponding author states that there is no conflict of interest.

\bibliographystyle{alpha}
\bibliography{Bibliography.bib}

\newcommand{\etalchar}[1]{$^{#1}$}
\begin{thebibliography}{KKMSD73}

\bibitem[ACM{\etalchar{+}}16]{AbramovichEtAlSkeletons}
D.~Abramovich, Q.~Chen, S.~Marcus, M.~Ulirsch, and J.~Wise.
\newblock Skeletons and fans of logarithmic structures.
\newblock In {\em Nonarchimedean and tropical geometry}, Simons Symp., pages
  287--336. Springer, 2016.

\bibitem[ACMW17]{AbramovichChenMarcusWise}
D.~Abramovich, Q.~Chen, S.~Marcus, and J.~Wise.
\newblock Boundedness of the space of stable logarithmic maps.
\newblock {\em J. Eur. Math. Soc. (JEMS)}, 19(9):2783--2809, 2017.

\bibitem[ACP15]{AbramovichCaporasoPayne}
D.~Abramovich, L.~Caporaso, and S.~Payne.
\newblock The tropicalization of the moduli space of curves.
\newblock {\em Ann. Sci. \'{E}c. Norm. Sup\'{e}r. (4)}, 48(4):765--809, 2015.

\bibitem[AK00]{AbramovichKaru}
D.~Abramovich and K.~Karu.
\newblock Weak semistable reduction in characteristic 0.
\newblock {\em Invent. Math.}, 139(2):241--273, 2000.

\bibitem[ALT20]{ALT-Semistable}
K.~Adiprasito, G.~Liu, and M.~Temkin.
\newblock Semistable reduction in characteristic 0.
\newblock {\em S\'{e}m. Lothar. Combin.}, 82B:Art. 25, 10, 2020.

\bibitem[AW18]{AbramovichWiseBirational}
D.~Abramovich and J.~Wise.
\newblock Birational invariance in logarithmic {G}romov-{W}itten theory.
\newblock {\em Compos. Math.}, 154(3):595--620, 2018.

\bibitem[BC23]{BattistellaCarocci}
L.~Battistella and F.~Carocci.
\newblock A smooth compactification of the space of genus two curves in
  projective space: via logarithmic geometry and {G}orenstein curves.
\newblock {\em Geom. Topol.}, 27(3):1203--1272, 2023.

\bibitem[BGS11]{BurgosGilSombra}
J.~I. Burgos~Gil and M.~Sombra.
\newblock When do the recession cones of a polyhedral complex form a fan?
\newblock {\em Discrete Comput. Geom.}, 46(4):789--798, 2011.

\bibitem[BNR22]{BNR2}
L.~{Battistella}, N.~{Nabijou}, and D.~{Ranganathan}.
\newblock {Gromov-Witten theory via roots and logarithms}.
\newblock {\em arXiv e-prints}, March 2022.
\newblock arXiv:2203.17224. \emph{Geom. Topol.}, to appear.

\bibitem[Bro14]{BrownToric}
J.~Brown.
\newblock Gromov-{W}itten invariants of toric fibrations.
\newblock {\em Int. Math. Res. Not. IMRN}, (19):5437--5482, 2014.

\bibitem[Bro21]{Brochard}
S.~Brochard.
\newblock Duality for commutative group stacks.
\newblock {\em Int. Math. Res. Not. IMRN}, (3):2321--2388, 2021.

\bibitem[CCUW20]{CavalieriChanUlirschWise}
R.~Cavalieri, M.~Chan, M.~Ulirsch, and J.~Wise.
\newblock A moduli stack of tropical curves.
\newblock {\em Forum Math. Sigma}, 8:Paper No. e23, 93, 2020.

\bibitem[CIJ18]{CoatesIritaniJiang}
T.~Coates, H.~Iritani, and Y.~Jiang.
\newblock The crepant transformation conjecture for toric complete
  intersections.
\newblock {\em Adv. Math.}, 329:1002--1087, 2018.

\bibitem[CLS11]{CLS}
D.~A. Cox, J.~B. Little, and H.~K. Schenck.
\newblock {\em Toric varieties}, volume 124 of {\em Graduate Studies in
  Mathematics}.
\newblock American Mathematical Society, Providence, RI, 2011.

\bibitem[CN24]{CN21}
F.~Carocci and N.~Nabijou.
\newblock Rubber tori in the boundary of expanded stable maps.
\newblock {\em J. Lond. Math. Soc. (2)}, 109(3):Paper No. e12874, 36, 2024.

\bibitem[Cox95a]{CoxFunctor}
D.~A. Cox.
\newblock The functor of a smooth toric variety.
\newblock {\em Tohoku Math. J. (2)}, 47(2):251--262, 1995.

\bibitem[Cox95b]{CoxHomogeneous}
D.~A. Cox.
\newblock The homogeneous coordinate ring of a toric variety.
\newblock {\em J. Algebraic Geom.}, 4(1):17--50, 1995.

\bibitem[FR16]{FosterRanganathan}
T.~Foster and D.~Ranganathan.
\newblock Degenerations of toric varieties over valuation rings.
\newblock {\em Bull. Lond. Math. Soc.}, 48(5):835--847, 2016.

\bibitem[Ful93]{FultonToric}
W.~Fulton.
\newblock {\em Introduction to toric varieties}, volume 131 of {\em Annals of
  Mathematics Studies}.
\newblock Princeton University Press, Princeton, NJ, 1993.
\newblock The William H. Roever Lectures in Geometry.

\bibitem[Hal03]{halic2003families}
M.~Halic.
\newblock Families of toric varieties.
\newblock {\em Mathematische Nachrichten}, 261(1):60--84, 2003.

\bibitem[HKM20]{HKM}
J.~{Hofscheier}, A.~{Khovanskii}, and L.~{Monin}.
\newblock {Cohomology rings of toric bundles and the ring of conditions}.
\newblock {\em arXiv e-prints}, June 2020.
\newblock arXiv:2006.12043.

\bibitem[HLY02]{HuLiuYau}
Y.~Hu, C.-H. Liu, and S.-T. Yau.
\newblock Toric morphisms and fibrations of toric {C}alabi-{Y}au hypersurfaces.
\newblock {\em Adv. Theor. Math. Phys.}, 6(3):457--506, 2002.

\bibitem[HMP{\etalchar{+}}22]{HMPPS}
D.~{Holmes}, S.~{Molcho}, R.~{Pandharipande}, A.~{Pixton}, and J.~{Schmitt}.
\newblock {Logarithmic double ramification cycles}.
\newblock {\em arXiv e-prints}, July 2022.
\newblock arXiv:2207.06778.

\bibitem[HPS19]{HolmesPixtonSchmitt}
D.~Holmes, A.~Pixton, and J.~Schmitt.
\newblock Multiplicativity of the double ramification cycle.
\newblock {\em Doc. Math.}, 24:545--562, 2019.

\bibitem[Hu08]{hutoric}
Y.~Hu.
\newblock Toric degenerations of {GIT} quotients, {C}how quotients, and
  {$\overline M_{0,N}$}.
\newblock {\em Asian J. Math.}, 12(1):47--53, 2008.

\bibitem[KKMSD73]{KKMSD}
G.~Kempf, F.~F. Knudsen, D.~Mumford, and B.~Saint-Donat.
\newblock {\em Toroidal embeddings. {I}}.
\newblock Lecture Notes in Mathematics, Vol. 339. Springer-Verlag, Berlin-New
  York, 1973.

\bibitem[KM19]{KavehManon}
K.~{Kaveh} and C.~{Manon}.
\newblock {Toric flat families, valuations, and applications to projectivized
  toric vector bundles}.
\newblock {\em arXiv e-prints}, July 2019.
\newblock arXiv:1907.00543.

\bibitem[Mol21]{MolchoSS}
S.~Molcho.
\newblock Universal stacky semistable reduction.
\newblock {\em Israel J. Math.}, 242(1):55--82, 2021.

\bibitem[MPS23]{MolchoPandharipandeSchmitt}
S.~Molcho, R.~Pandharipande, and J.~Schmitt.
\newblock The {H}odge bundle, the universal 0-section, and the log {C}how ring
  of the moduli space of curves.
\newblock {\em Compos. Math.}, 159(2):306--354, 2023.

\bibitem[MR20]{MandelRuddat}
T.~Mandel and H.~Ruddat.
\newblock Descendant log {G}romov-{W}itten invariants for toric varieties and
  tropical curves.
\newblock {\em Trans. Amer. Math. Soc.}, 373(2):1109--1152, 2020.

\bibitem[MR21]{MolchoRanganathan}
S.~{Molcho} and D.~{Ranganathan}.
\newblock {A case study of intersections on blowups of the moduli of curves}.
\newblock {\em arXiv e-prints}, June 2021.
\newblock arXiv:2106.15194. \emph{Alg. Num. Th.}, to appear.

\bibitem[MR24]{MR20}
D.~Maulik and D.~Ranganathan.
\newblock Logarithmic {D}onaldson--{T}homas theory.
\newblock {\em Forum Math. Pi}, 12:Paper No. e9, 2024.

\bibitem[Mum72]{MumfordAbelian}
D.~Mumford.
\newblock An analytic construction of degenerating abelian varieties over
  complete rings.
\newblock {\em Compositio Math.}, 24:239--272, 1972.

\bibitem[NR22]{MaxContacts}
N.~Nabijou and D.~Ranganathan.
\newblock Gromov-{W}itten theory with maximal contacts.
\newblock {\em Forum Math. Sigma}, 10:Paper No. e5, 34, 2022.

\bibitem[NS06]{NishinouSiebert}
T.~Nishinou and B.~Siebert.
\newblock Toric degenerations of toric varieties and tropical curves.
\newblock {\em Duke Math. J.}, 135(1):1--51, 2006.

\bibitem[Oh21]{OhGIT}
J.~Oh.
\newblock Quasimaps to {GIT} fibre bundles and applications.
\newblock {\em Forum Math. Sigma}, 9:Paper No. e56, 39, 2021.

\bibitem[{Ran}19]{RanganathanProducts}
D.~{Ranganathan}.
\newblock {A note on cycles of curves in a product of pairs}.
\newblock {\em arXiv e-prints}, October 2019.
\newblock arXiv:1910.00239.

\bibitem[Ran22]{DhruvExpansions}
D.~Ranganathan.
\newblock Logarithmic {G}romov-{W}itten theory with expansions.
\newblock {\em Algebr. Geom.}, 9(6):714--761, 2022.

\bibitem[RSPW19a]{RanganathanSantosParkerWise1}
D.~Ranganathan, K.~Santos-Parker, and J.~Wise.
\newblock Moduli of stable maps in genus one and logarithmic geometry, {I}.
\newblock {\em Geom. Topol.}, 23(7):3315--3366, 2019.

\bibitem[RSPW19b]{RanganathanSantosParkerWise2}
D.~Ranganathan, K.~Santos-Parker, and J.~Wise.
\newblock Moduli of stable maps in genus one and logarithmic geometry, {II}.
\newblock {\em Algebra Number Theory}, 13(8):1765--1805, 2019.

\bibitem[SU03]{SankaranUma}
P.~Sankaran and V.~Uma.
\newblock Cohomology of toric bundles.
\newblock {\em Comment. Math. Helv.}, 78(3):540--554, 2003.

\end{thebibliography}

\footnotesize
\noindent Francesca Carocci. EPFL, Switzerland. \href{mailto:francesca.carocci@epfl.ch}{francesca.carocci@epfl.ch}

\noindent Navid Nabijou. Queen Mary University of London, UK. \href{mailto:n.nabijou@qmul.ac.uk}{n.nabijou@qmul.ac.uk}

\end{document}